\documentclass[11pt, a4paper]{amsart}

\usepackage{amsmath, amssymb, amsthm}
\usepackage[english]{babel}
\usepackage[all]{xy}
\usepackage{graphicx}
\usepackage{tikz}
\usetikzlibrary{matrix}
\usepackage{verbatim}
\usepackage{enumerate}
\usepackage{url}
\usepackage{subfig}
\usepackage{bm}
\usepackage{xcolor}

\newtheorem{thm}{Theorem}[section]
\newtheorem{lem}[thm]{Lemma}

\newtheorem{prop}[thm]{Proposition}
\theoremstyle{definition}
\newtheorem{defn}[thm]{Definition}
\theoremstyle{remark}
\newtheorem{rem}[thm]{Remark}

\newcommand{\imagescaling}{1}
\newcommand{\note}[1]{}

\newcommand{\PH}{\phantom{\Sp}}
\newcommand{\dev}{\delta}

\newcommand{\tGt}{\tilde G^{(2)}}
\newcommand{\tGf}{\tilde G^{(4)}}

\newcommand{\CS}{\K}
\newcommand{\HC}[2]{[#1\to#2]}

\newcommand{\Cyc}{\textsf{\textbf{C}}}
\newcommand{\Net}{\textsf{\textbf{N}}}

\newcommand{\Mfld}{\mathcal{M}}

\newcommand{\Tornmo}{\Tor^{\maxdim-1}}
\newcommand{\Torm}{\Tor^{\maxpop}}

\newcommand{\ph}{\frac{\pi}{2}}

\newcommand{\al}{\alpha}

\newcommand{\h}{h}
\newcommand{\hloc}{\h^{\textrm{loc}}}
\newcommand{\hglob}{\h^{\textrm{gl}}}

\newcommand{\Hin}{H^{\textrm{in}}}
\newcommand{\Hout}{H^{\textrm{out}}}
\newcommand{\ve}[2]{\left(\begin{array}{c}#1\\#2\end{array}\right)}

\newcommand{\Wu}{W^\textnormal{u}}
\newcommand{\Ws}{W^\textnormal{s}}

\newcommand{\SDpp}{\Sp\Dp\psi_3\psi_4}
\newcommand{\SDpD}{\Sp\Dp\psi_3\Dp}
\newcommand{\SDDp}{\Sp\Dp\Dp\psi_4}
\newcommand{\SDDD}{\Sp\Dp\Dp\Dp}

\newcommand{\sfa}{\textsc{a}}
\newcommand{\sfb}{\textsc{b}}

\newcommand{\si}{\boldsymbol{\sigma}}

\newcommand{\lft}[1]{#1}
\newcommand{\Ph}{\lft{P}}
\newcommand{\Lh}{\lft{L}}
\newcommand{\Th}{\lft{T}}
\newcommand{\Tg}{\mathrm{T}}

\newcommand{\Dal}{\Delta\alpha}

\newcommand{\M}{\textit{\textsf{M}}}

\newcommand{\cA}{\textnormal{(A)}}
\newcommand{\cB}{\textnormal{(B)}}
\newcommand{\cC}{\textnormal{(C)}}

\newcommand{\maxdim}{N}
\newcommand{\maxpop}{M}

\newcommand{\Ctc}{\check\Cyc_2}
\newcommand{\Cth}{\hat\Cyc_2}

\newcommand{\sic}{\check\si}
\newcommand{\sih}{\hat\si}

\DeclareMathOperator{\sgn}{sgn}

\DeclareMathOperator{\Fix}{Fix}

\newcommand{\Splay}{\Dp}
\newcommand{\Sync}{\Sp}

\newcommand{\vth}{\vartheta}

\newcommand{\itw}[1]{#1_2}
\newcommand{\ifo}[1]{#1_4}

\newcommand{\aaa}{\alpha}
\newcommand{\at}{\itw{\aaa}}
\newcommand{\af}{\ifo{\aaa}}

\newcommand{\ud}{\mathrm{d}}
\newcommand{\udi}{\,\ud}

\DeclareMathOperator{\Rep}{Re}

\renewcommand{\Re}{\Rep}

\newcommand{\R}{\mathbb{R}}
\newcommand{\Rn}{\R^d}

\newcommand{\Z}{\mathbb{Z}}
\newcommand{\N}{\mathbb{N}}

\newcommand{\Zm}{\Z_\maxpop}

\newcommand{\lmax}{\lambda^{\max}}
\newcommand{\wmax}{u^\textnormal{max}}

\newcommand{\Tor}{\mathbf{T}}
\newcommand{\Torn}{\Tor^\maxdim}

\newcommand{\Tormn}{\Tor^{\maxpop\maxdim}}

\newcommand{\g}{\gamma}
\newcommand{\G}{\Gamma}

\newcommand{\Ss}{\mathbf{S}}

\newcommand{\Sn}{\Ss_\maxdim}

\newcommand{\abs}[1]{\left|#1\right|}

\newcommand{\rset}[2]{\left\lbrace\, #1\,\left|\;#2\right.\right\rbrace}
\newcommand{\lset}[2]{\left\lbrace\left. #1\;\right|\,#2\,\right\rbrace}
\newcommand{\set}[2]{\rset{#1}{#2}}

\newcommand{\sset}[1]{\left\lbrace #1\right\rbrace}

\newcommand{\Sp}{\mathrm{S}}
\newcommand{\Dp}{\mathrm{D}}

\newcommand{\DSS}{{\Dp\Sp\Sp}}
\newcommand{\SDS}{{\Sp\Dp\Sp}}
\newcommand{\SSD}{{\Sp\Sp\Dp}}
\newcommand{\DDS}{{\Dp\Dp\Sp}}
\newcommand{\SDD}{{\Sp\Dp\Dp}}
\newcommand{\DSD}{{\Dp\Sp\Dp}}

\newcommand{\SSSS}{{\Sp\Sp\Sp\Sp}}

\newcommand{\DSSS}{{\Dp\Sp\Sp\Sp}}
\newcommand{\SDSS}{{\Sp\Dp\Sp\Sp}}
\newcommand{\SSDS}{{\Sp\Sp\Dp\Sp}}
\newcommand{\DDSS}{{\Dp\Dp\Sp\Sp}}
\newcommand{\SDDS}{{\Sp\Dp\Dp\Sp}}
\newcommand{\DSDS}{{\Dp\Sp\Dp\Sp}}
\newcommand{\SSSD}{{\Sp\Sp\Sp\Dp}}
\newcommand{\DSSD}{{\Dp\Sp\Sp\Dp}}
\newcommand{\SDSD}{{\Sp\Dp\Sp\Dp}}

\newcommand{\po}{\psi_1}
\newcommand{\pt}{\psi_2}
\newcommand{\pr}{\psi_3}
\newcommand{\pf}{\psi_4}

\newcommand{\pppS}{{\po\pt\pr\Sp}}
\newcommand{\ppSp}{{\po\pt\Sp\pf}}

\newcommand{\DpS}{{\Dp\psi\Sp}}
\newcommand{\pDS}{{\psi\Dp\Sp}}

\newcommand{\DSp}{{\Dp\Sp\psi}}

\newcommand{\K}{K}
\newcommand{\B}{\mathcal B}

\newcommand{\ind}{\textnormal{ind}}

\begin{document}



\title[Heteroclinic Networks of Localized Frequency Synchrony]{Heteroclinic Dynamics of Localized Frequency Synchrony: Stability of Heteroclinic Cycles and Networks}
\author[Christian Bick and Alexander Lohse]{Christian Bick${}^\blacktriangledown$ and Alexander Lohse${}^\blacktriangle$}%

\address{${}^\blacktriangledown$Centre for Systems Dynamics and Control and Department of Mathematics, University of Exeter, EX4~4QF, UK}
\address{${}^\blacktriangle$Universit\"at Hamburg, Fachbereich Mathematik, Bundesstra{\ss}e 55, 20146 Hamburg, Germany}
\date{\today}

\maketitle

\begin{abstract}
In the first part of this paper, we showed that three coupled populations of identical phase oscillators give rise to heteroclinic cycles between invariant sets where populations show distinct frequencies. Here, we now give explicit stability results for these heteroclinic cycles for populations consisting of two oscillators each. In systems with four coupled phase oscillator populations, different heteroclinic cycles can form a heteroclinic network. While such networks cannot be asymptotically stable, the local attraction properties of each cycle in the network can be quantified by stability indices. We calculate these stability indices in terms of the coupling parameters between oscillator populations. Hence, our results elucidate how oscillator coupling influences sequential transitions along a heteroclinic network where individual oscillator populations switch sequentially between a high and a low frequency regime; such dynamics appear relevant for the functionality of neural oscillators.
\end{abstract}

\section{Introduction}

Interacting populations of identical oscillators are capable of generating global dynamics that exhibit rapid transitions between metastable states where different populations are in different frequency regimes. Such dynamics can be caused by trajectories close to heteroclinic structures between invariant sets where frequency synchrony is local rather than global across all populations~\cite{Bick2017c}. In the first part of this paper~\cite{Bick2018a}, we showed the existence of heteroclinic cycles in three coupled small populations.

For such dynamics to be observable over longer timescales, the heteroclinic cycles have to have some stability properties. Apart from asymptotic stability and instability, heteroclinic cycles can display various intermediate forms of nonasymptotic attraction. These range from fragmentary asymptotic stability (``attracting more than nothing'') to essential asymptotic stability (``attracting almost everything''). Podvigina and Ashwin~\cite{Podvigina2011} introduced a stability index to quantify attraction along trajectories. This stability index is defined for any dynamically invariant set and thus provides a convenient tool to describe the stability of heteroclinic trajectories within a cycle or network\footnote{To avoid confusion in terminology, we reserve the word ``network'' for heteroclinic networks and talk about coupled (or interacting) populations of phase oscillators (rather than oscillator networks).}. Recently, Garrido-da-Silva and Castro~\cite{Garrido-da-Silva2016} derived explicit expressions for the stability indices for a fairly general class of heteroclinic cycles called quasi-simple. Such expressions are particularly useful to describe the stability of heteroclinic cycles that are part of a networks consisting of more than one cycle.

The main contributions of this paper are explicit stability results for heteroclinic cycles and networks between invariant sets with localized frequency synchrony in terms of the coupling parameter of the oscillator populations. Here we focus on coupled oscillator populations with two oscillators per population. Due to the existence of invariant subspaces, the heteroclinic cycles  are quasi-simple. Consequently, we apply the stability results of Garrido-da-Silva and Castro~\cite{Garrido-da-Silva2016} to calculate the stability indices. We first consider three coupled oscillator populations to calculate stability indices for the heteroclinic cycles in~\cite{Bick2018a}. We then show that four coupled oscillator populations support a heteroclinic network which contains two distinct heteroclinic cycles of the type considered before. Their stability properties are then calculated using the tools developed for three populations and we comment on the stability of the whole network. Since our stability conditions explicitly depend on the coupling parameters of the oscillator populations, our results elucidate how the coupling structure of the system shapes the asymptotic dynamical behavior. Moreover, they highlight the utility of the general stability results in~\cite{Garrido-da-Silva2016} for heteroclinic cycles on arbitrary manifolds.

The remainder of this paper is structured as follows. The following section summarizes facts on (robust) heteroclinic cycles, nonasymptotic stability, and coupled populations of phase oscillators. In Section~\ref{sec:cycles} we calculate the stability indices along the heteroclinic cycle in the first part of this paper~\cite{Bick2018a} for a system of three populations. Such cycles are contained in a heteroclinic network for four coupled populations as shown in Section~\ref{sec:networks}, and we calculate their stability properties. We also give some numerical results and comment on the stability of the network as a whole. Finally, we give some concluding remarks in Section~\ref{sec:disc}.

\section{Preliminaries}
\label{sec:prelim}

To set the stage, we review some results about heteroclinic cycles, their stability properties, and coupled populations of phase oscillators. In terms of notation, we will follow the first part of the paper~\cite{Bick2018a}.

\subsection{Heteroclinic cycles and their stability}

Let~$\Mfld$ be a smooth $d$-dimensional manifold and let~$X$ be a smooth vector field on~$\Mfld$. Define the usual limit sets~$\alpha(x)$, $\omega(x)$ for the flow on~$\Mfld$ generated by~$X$ and $t\to\pm\infty$. For a hyperbolic equilibrium $\xi\in\Mfld$ we write
\begin{align*}
\Ws(\xi) &:= \set{x\in\Mfld}{\omega(x)=\xi}, & \Wu(\xi) &:= \set{x\in\Mfld}{\alpha(x)=\xi}
\end{align*}
to denote its stable and unstable manifold, respectively.

\begin{defn}\label{defn:HetCycle}
A \emph{heteroclinic cycle~$\Cyc$} consists of a finite number of hyperbolic equilibria~$\xi_q\in\Mfld$, $q=1,\dotsc,Q$, together with heteroclinic trajectories
\[\HC{\xi_q}{\xi_{q+1}} \subset \Wu(\xi_q)\cap \Ws(\xi_{q+1})\neq\emptyset\]
where indices are taken modulo~$Q$.

A \emph{heteroclinic network}~$\Net$ is a connected union of two or more distinct  heteroclinic cycles.
\end{defn}

For simplicity, we write $\Cyc=(\xi_1, \dotsc, \xi_Q)$. If~$\Mfld$ is a quotient of a higher-dimensional manifold and~$\Cyc$ is a heteroclinic cycle in~$\Mfld$, we also call the lift of~$\Cyc$ a heteroclinic cycle. The same goes for a heteroclinic network~$\Net$.

While heteroclinic cycles are in general a nongeneric phenomenon, they can be robust if all connections are of saddle-sink type in (lower-dimensional) subspaces. Let $\Cyc=(\xi_1, \dotsc, \xi_Q)$ be a heteroclinic cycle. If there are flow-invariant submanifolds~$P_q\subset\Mfld$ such that $\HC{\xi_q}{\xi_{q+1}}\subset P_q$ is a saddle-sink connection, then~$\Cyc$ is \emph{robust} with respect to perturbations of~$X$ which preserve the invariant sets~$P_q$. 

Robust heteroclinic cycles may arise for example in dynamical systems with symmetry. Let~$\G$ be a finite group which acts on~$\Mfld$. For a subgroup $H\subset\G$ define the set $\Fix(H) = \set{x\in\Mfld}{\g x=x\ \forall\g\in H}$  of points fixed under~$H$; this is a vector space that is invariant under the flow generated by~$X$. For $x\in\Mfld$ let $\Gamma x=\set{\g x}{\g\in\G}$ denote its group orbit and $\Sigma(x) = \set{\g\in\G}{\g x = x}$ its \emph{isotropy subgroup}. Now assume that the smooth vector field~$X$ is a $\G$-equivariant vector field on~$\Mfld$, that is, the action of the group commutes with~$X$. Any heteroclinic cycle with~$P_q = \Fix(\Sigma_q)$ where~$\Sigma_q$ are isotropy subgroups is robust to $\Gamma$-equivariant perturbations of~$X$, that is, $\G$-equivariant vector fields close to~$X$ will have a heteroclinic cycle close to~$\Cyc$; see~\cite{Krupa1997} for more details.

\subsubsection{Nonasymptotic stability}
Heteroclinic cycles may have intricate nonasymptotic stability properties. We briefly recall some definitions that formalize these.

For $\varepsilon>0$, write $B_{\varepsilon}(A)$ for an $\varepsilon$-neighborhood of a set $A \subset \Rn$ and~$\B(A)$ for its basin of attraction, i.e., the set of points $x \in \Rn$ with $\omega(x) \subset A$. For $\delta>0$ the $\delta$-local basin of attraction is
\begin{equation*}
\B_\delta(A):=\set{x \in \B(A)}{\forall t>0\colon \Phi_t(x) \in B_\delta(A)},
\end{equation*}
where~$\Phi_t$ is the flow generated by~$X$. Let~$\ell$ denote the Lebesgue measure.

\begin{defn}[\cite{Podvigina2012}]
An invariant set~$A$ is \emph{fragmentarily asymptotically stable (f.a.s.)} if $\ell({\mathcal B}_\delta(A))>0$ for any $\delta>0$.
\end{defn}

Being f.a.s.\ is not necessarily a very strong form of attraction. A set that is not f.a.s.\ is usually called \emph{completely unstable}, see also \cite{Podvigina2012}. Melbourne \cite{Melbourne1991} introduces the stronger notion of essential asymptotic stability, which we quote here in the formulation of Brannath~\cite{Brannath1994}.

\begin{defn}[\cite{Brannath1994}, Definition 1.2]
A compact invariant set~$A$ is called {\em essentially asymptotically stable (e.a.s.)} if it is asymptotically stable relative to a set $\Upsilon \subset \Rn$ with the property that
\begin{equation}
\lim\limits_{\varepsilon \to 0} \frac{\ell(B_{\varepsilon}(A) \cap \Upsilon)}{\ell(B_\varepsilon(A))} = 1.\label{pas-property}
\end{equation}
\end{defn}

Podvigina and Ashwin~\cite{Podvigina2011} introduced the concept of a local stability index $\si(x)\in[-\infty,+\infty]$ to quantify stability and attraction. It is constant along trajectories, so to characterize stability/attraction of a heteroclinic cycle with one-dimensional connections, it suffices to consider finitely many stability indices. Let~$\si_q$ denote the stability index along $\HC{\xi_{q-1}}{\xi_{q}}$. For our purposes it is enough to note that (under some mild assumptions) a heteroclinic cycle $\Cyc=(\xi_1, \dotsc,\xi_Q)$ is completely unstable if $\si_q=-\infty$ for all~$q$, it is f.a.s.\ as soon as $\si_q>-\infty$ for some~$q$, and it is e.a.s.\ if and only if $\si_q>0$ for all $q=1,\dotsc,Q$. See \cite[Theorem 3.1]{Lohse2015} for details.

\subsubsection{Stability of quasi-simple heteroclinic cycles}

The stability indices can be calculated for specific classes of heteroclinic cycles. Let $\Cyc=(\xi_1, \dotsc, \xi_Q)$ be a robust heteroclinic cycle on~$\Mfld$. As above, denote the flow-invariant sets which contain the heteroclinic connections with~$P_q$. Let~$\Th_q := \Tg_{\xi_q}\Mfld$ denote the tangent space of~$\Mfld$ at~$\xi_q$. For subspaces $V\subset W\subset \Th_q$ write $W\ominus V$ for the orthogonal complement of~$V$ in~$W$. In slight abuse of notation, define~$\Ph_q^- := \Tg_{\xi_q} P_{q-1}$ and $\Ph_q^{+} := \Tg_{\xi_q} P_{q}$ to be the tangent spaces of~$P_{q-1}$ and~$P_q$ in~$\Th_q$, respectively. These are linear subspaces of~$T_q$ of the same dimension as~$P_{q-1}$ (which contains the incoming saddle connection) and~$P_q$ (containing the outgoing connection), respectively. Set $\Lh_q := \Ph_{q}^- \cap \Ph_q^+$.

\begin{defn}
\label{def:QuasiSimple}
The robust heteroclinic cycle~$\Cyc$ is \emph{quasi-simple} if
 $\dim(\Ph_q^- \ominus \Lh_q)=\dim(\Ph_q^+ \ominus \Lh_q)=1$\label{cond:Contr}
 for all~$q\in\sset{1, \dotsc, Q}$.
\end{defn}

\begin{rem}
Note that this is a slight generalization of the definition given by Garrido-da-Silva and Castro in~\cite{Garrido-da-Silva2016} to arbitrary manifolds. In particular, the condition in Definition~\ref{def:QuasiSimple} implies that $\dim(P_{q-1})=\dim(P_q)$.
\end{rem}

As usual, an eigenvalue of the Jacobian~$\ud X(\xi_q)$ is \emph{radial} if its associated eigenvector is in~$\Lh_q$, \emph{contracting} if the associated eigenvector is in $\Ph_{q}^-\ominus\Lh_q$, \emph{expanding} if the associated eigenvector is in $\Ph_{q}^+\ominus\Lh_q$, and \emph{transverse} otherwise. In other words, a cycle is quasi-simple  if it has unique expanding and contracting directions at each equilibrium, and thus one-dimensional saddle connections.

The standard way to analyze the stability of heteroclinic cycles is to write down a Poincar\'e return map with linearized dynamics local to the equilibria as well as globally along the connecting orbits; cf.~\cite{Krupa1995}. For quasi-simple cycles whose global maps are rescaled permutations of the local coordinate axes Garrido-da-Silva and Castro~\cite{Garrido-da-Silva2016} showed how their (asymptotic or nonasymptotic) stability can be calculated solely from the properties of the linearization of the equilibria at the cycle. More precisely, the stability of each equilibrium~$\xi_q$ along the cycle is encoded in a \emph{transition matrix}~$\M_q$ and the stability of the cycle is determined by properties of these matrices. We explain this technique in more detail when we apply it in Section~\ref{sec:cycles}. Note that this immediately implies that the results in~\cite{Garrido-da-Silva2016} carry over to our definition of a quasi-simple heteroclinic cycle since the stability does not depend on other global properties.

For ease of reference, we recall the stability results from~\cite[Theorems 3.4, 3.10]{Garrido-da-Silva2016} in a condensed form. For a heteroclinic cycle $\Cyc=(\xi_1,\dotsc,\xi_Q)$ with transition matrices~$\M_q$ set $\M^{(q)}:=\M_{q-1}\dotsb \M_1\M_Q\dotsb \M_{q+1}\M_q$. All~$\M^{(q)}$ have the same eigenvalues. If none of the~$\M_q$ has a negative entry---there are no repelling transverse directions---we have the following result, which is a dichotomy between asymptotic stability and complete instability.

\begin{prop}[{\cite[Theorem~3.4]{Garrido-da-Silva2016}}]\label{thm:3.4}
Let~$\Cyc$ be a quasi-simple heteroclinic cycle with rescaled permutation of the local coordinate axes as global maps and transition matrices~$\M_q$, $q=1,\dotsc,Q$.
Suppose that all entries of all~$\M_q$ are nonnegative.
  \begin{enumerate}[(i)]
  \item If $\M^{(1)}$ satisfies $\abs{\lmax}>1$, then $\si_q=+\infty$ for all $q=1,\dotsc,Q$ and the cycle~$\Cyc$ is asymptotically stable.
  \item Otherwise, $\si_q=-\infty$ for all $q=1,\dotsc,Q$ and~$\Cyc$ is completely unstable.
  \end{enumerate}
\end{prop}

If the transition matrices~$\M_q$ contain negative entries---there are transversely repelling directions, for example, if the cycle is part of a network---then additional criteria have to be satisfied in order for the cycle to possess some form of nonasymptotic  stability. For a matrix~$\M$ let~$\lmax$ denote the maximal eigenvalue and $\wmax=(\wmax_1,\dotsc,\wmax_d)$ the corresponding eigenvector. Define the conditions (cf.~\cite[Lemma 3.2]{Garrido-da-Silva2016})
\begin{itemize}
  \item[\cA] $\lmax$ is real.
  \item[\cB] $\lmax>1$.
  \item[\cC] $\wmax_m \wmax_n>0$ for all $m,n=1, \dotsc, d$.
\end{itemize}
Generally, stability indices are evaluated as a function of the local stability properties at the equilibrium points~\cite{Podvigina2011}; for quasi-simple cycles in arbitrary dimension, Garrido-da-Silva and Castro~\cite{Garrido-da-Silva2016} denote this function by~$F^{\ind}$. Later on, we will consider three-dimensional transition matrices and for $0\neq \beta=(\beta_1, \beta_2, \beta_3) \in \R^3$, this function reads
\begin{align*}
F^{\ind}(\beta):=
\begin{cases}
 +\infty \quad & 
 \begin{array}{l}
 \textnormal{if }\min(\beta_1,\beta_2,\beta_3)\geq 0,
 \end{array}\\
 -\frac{\beta_1+\beta_2+\beta_3}{\min(\beta_1,\beta_2,\beta_3)} \quad &
 \begin{array}{l}
 \textnormal{if }\beta_1+\beta_2+\beta_3>0\\
 \textnormal{\quad and }\min(\beta_1,\beta_2,\beta_3)<0,
 \end{array}\\
 0 \quad &
 \begin{array}{l}
 \textnormal{if}\ \beta_1+\beta_2+\beta_3=0,
 \end{array}\\
 \frac{\beta_1+\beta_2+\beta_3}{\max(\beta_1,\beta_2,\beta_3)} \quad & 
 \begin{array}{l}
 \textnormal{if}\ \beta_1+\beta_2+\beta_3<0\\ 
 \quad\textnormal{ and }\max(\beta_1,\beta_2,\beta_3)>0,
 \end{array}\\
  -\infty \quad &
 \begin{array}{l}  
  \textnormal{if }\max(\beta_1,\beta_2,\beta_3)\leq 0.
 \end{array}\\  
\end{cases}
\end{align*}
The following proposition summarizes the second stability result adapted to our setting.

\begin{prop}[{\cite[Theorem~3.10]{Garrido-da-Silva2016}}]\label{thm:3.10}
Let~$\Cyc$ be a quasi-simple heteroclinic cycle with rescaled permutation of the local coordinate axes as global maps and transition matrices~$\M_q$, $q=1,\dotsc,Q$. 
Suppose that at least one~$\M_q$ has at least one negative entry.
  \begin{itemize}
  \item[(a)] If there is at least one~$q$ such that the matrix~$\M^{(q)}$ does not satisfy conditions \cA--\cC, then $\si_q=-\infty$ for all $q=1,\dotsc,Q$ and~$\Cyc$ is completely unstable.
  \item[(b)] If all $\M^{(q)}$ satisfy conditions \cA--\cC, then~$\Cyc$ is f.a.s.\ and there exist $\beta^{(1)},\dotsc,\beta^{(s)} \in \R^3$ such that the stability indices for~$\Cyc$ are given by \[\si_q=\min\limits_{l=1,\dotsc,s}F^\ind\big(\beta^{(l)}\big).\]
    \end{itemize}
\end{prop}

\noindent Here~$s$ is bounded by an expression which depends on the number of rows of the transition matrices (and their products) with at least one negative entry; cf.~\cite{Garrido-da-Silva2016} for details.

\subsection{Coupled populations of phase oscillators}

Consider~$\maxpop$ populations of~$\maxdim$ phase oscillators where $\theta_{\sigma,k}\in\Tor:=\R/2\pi \Z$ denotes the phase of oscillator~$k$ in population~$\sigma$. Hence, the state of the coupled oscillator populations is determined by $\theta=(\theta_{1}, \dotsc, \theta_{\maxpop})\in\Tormn$ where $\theta_\sigma = (\theta_{\sigma, 1}, \dotsc, \theta_{\sigma, \maxdim})\in\Torn$ is the state of population~$\sigma$.  Let~$\Sn$ denote the permutation group of~$\maxdim$ elements. Suppose that the phase evolution is given by
\begin{equation}\label{eq:DynMxN}
\dot\theta_{\sigma,k} := \frac{\ud}{\ud t}\theta_{\sigma,k} = \omega + Y_{\sigma,k}(\theta)
\end{equation}
where~$\omega$ is the intrinsic frequency of each oscillator and the vector field~$Y$ is $(\Sn\times\Tor)^\maxpop$-equivariant. Here, each copy of~$\Tor$ acts by shifting all oscillator phases of a given population~$\sigma$ by a common constant while~$\Sn$ permutes the oscillator indices~$k$.

The symmetry implies that certain phase configurations are dynamically invariant.
For a single population of $\maxdim$ oscillators, the subset
\begin{align}
\Sync &:= \set{(\phi_1, \dotsc, \phi_\maxdim)\in\Torn}{\phi_k=\phi_{k+1}}
\intertext{corresponds to phases being in full phase synchrony and}
\Splay &:= \set{(\phi_1, \dotsc, \phi_\maxdim)\in\Torn}{\phi_{k+1}=\phi_{k}+\frac{2\pi}{\maxdim}}
\end{align}
denotes a splay phase configuration---typically we call any element of the group orbit~$\Sn\Splay$ a splay phase. For interacting oscillator populations, we use the shorthand notation
\begin{subequations}\label{eq:SyncSplay}
\begin{align}
\theta_1\dotsb\theta_{\sigma-1}\Sp\theta_{\sigma+1}\dotsb\theta_{\maxpop} &= \lset{\theta\in\Tormn}{\theta_\sigma\in\Sync}\\
\theta_1\dotsb\theta_{\sigma-1}\Dp\theta_{\sigma+1}\dotsb\theta_{\maxpop} &= \lset{\theta\in\Tormn}{\theta_\sigma\in\Splay}
\end{align}
\end{subequations}
to indicate that population~$\sigma$ is fully phase synchronized or in splay phase. Consequently, $\Sp\dotsb\Sp$ ($\maxpop$~times) is the set of cluster states where all populations are fully phase synchronized and $\Dp\dotsb\Dp$ the set where all populations are in splay phase. Because of the~$\Sn^\maxpop$ symmetry, the sets~\eqref{eq:SyncSplay} are invariant~\cite{Ashwin1992}.

To reduce the phase-shift symmetry~$\Torm$ we may rewrite~\eqref{eq:DynMxN} in terms of phase differences $\psi_{\sigma,k} := \theta_{\sigma, k+1} - \theta_{\sigma, 1}$, $k=1, \dotsc, \maxdim-1$. Hence, with $\psi_\sigma\in\Tornmo$ we also write for example $\psi_1\Sync\dotsb\Sync$ (or simply $\psi\Sync\dotsb\Sync$ if the index is obvious) to indicate that all but the first population is phase synchronized. The sets~\eqref{eq:SyncSplay} are equilibria relative to~$\Torm$, that is, they are equilibria for the reduced system in terms of phase differences.

\subsubsection{Frequencies and localized frequency synchrony}
Suppose that $\maxpop > 1$ and let~$\theta:[0,\infty)\to\Tormn$ be a solution of~\eqref{eq:DynMxN} with initial condition $\theta(0)=\theta^0$. While $\dot\theta_{\sigma,k}(t)$ is the \emph{instantaneous angular frequency} of oscillator $(\sigma, k)$, define the \emph{asymptotic average angular frequency} of oscillator $(\sigma, k)$ by \begin{equation}
\Omega_{\sigma,k}(\theta^0):=\lim_{T\to\infty}\frac{1}{T}\int_0^T\dot\theta_{\sigma,k}(t)\udi t.
\end{equation}
Here we assume that these limits exist for all oscillators but this notion can be generalized to frequency intervals; see also~\cite{Bick2015c, Bick2015d}.

\begin{defn}
A connected flow-invariant invariant set $A\subset\Tormn$ has \emph{localized frequency synchrony} if for any~$\theta^0\in A$ and fixed~$\sigma$ we have $\Omega_{\sigma,k} = \Omega_{\sigma}$ for all~$k$ and there exist indices $\sigma\neq\tau$ such that
\begin{align}
\Omega_{\sigma} \neq \Omega_{\tau}.
\end{align}
\end{defn}

\begin{rem}
Note that a chain-recurrent set~$A$ with localized frequency synchrony is a \emph{weak chimera} as defined in~\cite{Ashwin2014a}.
\end{rem}

\section{Three Coupled Oscillator Populations}
\label{sec:cycles}

Here we derive explicit stability results for the heteroclinic cycles in $\maxpop=3$ coupled populations of $\maxdim=2$ phase oscillators~\eqref{eq:DynMxN} considered in the first part~\cite{Bick2018a}; we use the same notation introduced there. Interactions between pairs of oscillators are mediated by the coupling function
\begin{equation}\label{eq:CN2}
g(\vth) = \sin(\vth+\alpha)-r\sin(a(\vth+\alpha)).
\end{equation}
With the interaction function
\begin{align}\label{eq:NPCN2}
\tGf(\theta_\tau; \vth) &=
-\frac{1}{4}\big(\cos(\theta_{\tau, 1}-\theta_{\tau,2}+\vth+\alpha)+\cos(\theta_{\tau, 2}-\theta_{\tau,1}+\vth+\alpha)\big),
\end{align}
the phase dynamics for coupling strength $K>0$ between populations are given by
\begin{align}\label{eq:Dyn3x2}
\begin{split}
\dot\theta_{\sigma,k} &= \omega+g(\theta_{\sigma,3-k}-\theta_{\sigma,k})
+\K\tGf(\theta_{\sigma-1}; \theta_{\sigma,3-k}-\theta_{\sigma,k})
\\&\qquad\qquad
-\K\tGf(\theta_{\sigma+1}; \theta_{\sigma,3-k}-\theta_{\sigma,k}),
\end{split}
\end{align}%
$\sigma\in\sset{1,2,3}$, $k\in\sset{1,2}$. These are the equations of motion\footnote{As shown in \cite{Bick2017c} and Appendix~\ref{app:Reduction} these equations arise from an approximation of coupled populations with state-dependent phase shift by nonpairwise coupling terms.} considered in the first part~\cite{Bick2018a} with phase shifts parametrized by $\al := \at = \af-\ph$.

The interactions between populations in~\eqref{eq:Dyn3x2}---which include nonpairwise coupling---are a special case of~\eqref{eq:DynMxN}. More precisely, with $\Zm:=\Z/\maxpop\Z$ the equations~\eqref{eq:Dyn3x2} are $(\Sn\times\Tor)^\maxpop\rtimes \Zm$-equivariant. Each copy of~$\Tor$ acts by shifting all oscillator phases of one population by a common constant while~$\Sn$ permutes its oscillators. The action of~$\Zm$ permutes the populations cyclically. These actions do not necessarily commute. The phase space of~\eqref{eq:Dyn3x2} is organized by invariant subspaces and there are relative equilibria $\DSS$, $\DDS$ and their images under the~$\Z_3$ action.

\subsection{Heteroclinic cycles and local stability}
The coupled oscillator populations~\eqref{eq:Dyn3x2} with interaction functions~\eqref{eq:CN2}, \eqref{eq:NPCN2} support a robust heteroclinic cycle~\cite{Bick2018a}. Linear stability of~$\DSS$,~$\DDS$ are given by the eigenvalues%
\begin{subequations}%
\label{eq:StabDSS}%
\begin{align}
\lambda^\DSS_1 &= 2\cos(\al)+4r\cos(2\al),\\
\lambda^\DSS_2 &= 2\K\sin(\al)-2\cos(\al)+4r\cos(2\al),\\
\lambda^\DSS_3 &= -2\K\sin(\al)-2\cos(\al)+4r\cos(2\al),
\end{align}
\end{subequations}
and 
\begin{subequations}\label{eq:StabDDS}
\begin{align}
\lambda^\DDS_1 &= 2\K\sin(\al)+2\cos(\al)+4r\cos(2\al),\\
\lambda^\DDS_2 &= -2\K\sin(\al)+2\cos(\al)+4r\cos(2\al),\\
\lambda^\DDS_3 &= -2\cos(\al)+4r\cos(2\al).
\end{align}
\end{subequations}

\begin{lem}[Lemma~3.2 in~\cite{Bick2018a}]\label{lem:HetCycleM3}
Suppose that
$\lambda^\DSS_3<0<\lambda^\DSS_2$ and
$\lambda^\DDS_2<0<\lambda^\DDS_1$.
Then the $\maxpop=3$ coupled populations of $\maxdim=2$ phase oscillators~\eqref{eq:Dyn3x2} with interaction functions~\eqref{eq:CN2}, \eqref{eq:NPCN2} have a robust heteroclinic cycle
\begin{equation*}\label{eq:HetCycle}
\Cyc_2 = (\DSS, \DDS, \SDS, \SDD, \SSD, \DSD, \DSS).
\end{equation*}
\end{lem}

For fixed~$\alpha\approx\frac{\pi}{2}$, the assumptions of Lemma~\ref{lem:HetCycleM3} define a cone-shaped region in $(\K,r)$ parameter space: there is an affine linear function~$L$ such that $K>K_0$ where $L(K_0)=0$ and~$r$ between~$-L(K)$ and~$L(K)$. For the remainder of this section, we assume that the assumptions of Lemma~\ref{lem:HetCycleM3} hold. 

\begin{lem}
The cycle $\Cyc_2$ is quasi-simple.
\end{lem}

\begin{proof}
It suffices to consider the equilibria~$\DSS$ and~$\DDS$ due to the symmetry which permutes populations. We have $\Wu(\DSS)\subset\DpS$, $\Wu(\DDS)\subset\pDS$ which implies that each saddle has one contracting, expanding, and transverse eigenvalue; there are no radial eigenvalues since $\DSp\cap\DpS = \DSS$ and $\DpS\cap\pDS = \DDS$.
\end{proof}

Subject to nonresonance conditions, we may linearize the flow around the equilibria; see also~\cite[Proposition~4.1]{Aguiar2010a}.

\begin{lem}\label{lem:Linearizability}
Suppose that $\lambda^\DSS_1, \lambda^\DDS_3\neq 0$, and
\begin{subequations}\label{eq:CondNonRes}
\begin{align}
0 &\neq 2r\cos(2\al)\pm 3\cos(\al),\\
0 &\neq 4\K\sin(\al)\pm 4r\cos(2\al)\pm 2\cos(\al)
\end{align}
\end{subequations}
(in the second line we allow any combination of $+$ and $-$). Then we can linearize the flow at the equilibria in~$\Cyc_2$. For $\al=\ph$, these conditions reduce to $r\neq 0$ and $r\neq\pm\K$.
\end{lem}

\begin{proof}
According to the~$C^1$ linearization theorem~\cite{Ruelle1989} we can linearize the flow if the eigenvalues~$\lambda_l$ of the linearization satisfy $\Re{\lambda_l} \neq \Re{\lambda_j} + \Re\lambda_k$ when $\Re{\lambda_j}<0<\Re{\lambda_k}$. Given~\eqref{eq:StabDSS}, conditions~\eqref{eq:CondNonRes} are just these nonresonance conditions. Plugging in $\al=\ph$ yields the second assertion.
\end{proof}

\subsection{Cross sections, transition matrices, and stability}
\label{sec:Stability}

Using standard notation, we write
\begin{align}
-c_\DSS &:= \lambda^\DSS_3, & e_\DSS &:= \lambda^\DSS_2, & t_\DSS &:= \lambda^\DSS_1,\\
-c_\DDS &:= \lambda^\DDS_2, & e_\DDS &:= \lambda^\DDS_1, & t_\DDS &:= \lambda^\DDS_3,
\end{align}
for the contracting, expanding, and transverse eigenvalues. Thus $e_q, c_q>0$, $q\in\sset{\DSS, \DDS}$. The ratios between contraction/transverse stability and expansion are given by
\begin{align}
a_q &:= \frac{c_q}{e_q}, & b_q &:= -\frac{t_q}{e_q}
\end{align}
for $q\in\sset{\DSS, \DDS}$; we have~$a_q>0$ by definition and $b_q>0$ if $t_q<0$.

\subsubsection{Poincar\'e map and transition matrices}
We first consider the linearized flow at the equilibria to calculate the local transition maps. Introduce local coordinates $(v,w,z)$ which correspond to the contracting, expanding, and transverse directions, respectively. After appropriate rescaling,  consider the cross sections
\begin{align}
\Hin_q &= \set{(v,w,z)}{\abs{v} = 1, \abs{w}\leq 1, \abs{z}\leq 1},\\
\Hout_q &= \set{(v,w,z)}{\abs{v}\leq 1, \abs{w} = 1, \abs{z}\leq 1}
\end{align}
at $q\in\sset{\DSS, \DDS}$. The linearized flow at~$\xi_q$ is
\[
\Phi^\tau_q(v, w, z) = \left(\exp(-c_q\tau)v, \exp(e_q\tau)w, \exp(t_q\tau)z\right).
\]
Hence the time of flight is $\tau=-\log(w)/e_q$ which implies that the local map at~$\xi_q$ is
\[
\hloc_q:\Hin_q\to\Hout_q, (\pm1,w,z)\mapsto 
(w^{a_q}, \pm1, w^{b_q}z)
\]

Considering the invariant subspaces, we see that the global maps are rescaled permutations. More specifically, we have 
\[\hglob_q:\Hout_q\to\Hin_{q+1}, (v,w,z)\mapsto(A_q w,B_q z,D_q v).\]

Write $\h_q := \hglob_q\circ\hloc_q: \Hin_q\to\Hin_{q+1}$. Ignoring~$v$, this yields a map between the incoming 2-dimensional sections of subsequent equilibria
\[
\h_q(w,z) = (B_q w^{b_q}z,D_q w^{a_q}).
\]
Taken together, the Poincar\'e return map for the linearized dynamics around the heteroclinic cycle (modulo the~$\Z_3$ group action) is
\[\h = \h_\DDS \circ \h_\DSS.\]

If we introduce logarithmic coordinates we can write the return map in terms of transition matrices~\cite{Field1991, Garrido-da-Silva2016}. Restrict to the $(w,z)$ coordinates and introduce logarithmic variables 
$\eta = \log(w)$, $\zeta = \log(z)$.
In the new variables, the maps~$\h_q$ become linear,
\begin{equation}
\hat\h_q(\eta, \zeta) = \M_q\ve{\eta}{\zeta} + \ve{\log B_q}{\log D_q}
\end{equation}
with
\begin{align}\label{2x2matrix}
\M_q &= \left(\begin{array}{cc}b_q & 1 \\a_q & 0\end{array}\right)
\end{align}
Note that these transition matrices are the same as the ones for simple cycles in~$\R^4$ of type~$C$~\cite{Podvigina2011}.

The transition matrix for the Poincar\'e map~$\h$ is $\M_\DDS \M_\DSS$. These transition matrices govern the stability of the cycle~\cite[Theorem~3.4]{Garrido-da-Silva2016}. 

\subsubsection{Stability of~$\Cyc_2$ for $\alpha=\frac{\pi}{2}$}
The stability properties at the saddles are symmetric and stability is governed by the properties of the transition matrix
\begin{align}
\M &= \left(\begin{array}{cc}b & 1 \\a & 0\end{array}\right).
\end{align}
Here we omitted the saddle index~$q$ since $\M_\DSS=\M_\DDS=\M$. This is the same transition matrix as for a simple heteroclinic cycle in~$\R^4$ of type~$C_1^-$~\cite{Podvigina2011}. 

\begin{lem}[{\cite[Section~4.2.2]{Podvigina2011}}]
\label{lem:StabC1m}
A heteroclinic cycle whose stability is given by the transition matrix~$\M$ is asymptotically stable if $b \geq 0$ (that is, $t \leq 0$) and $a+b>1$; otherwise it is completely unstable.
\end{lem}

In terms of the oscillator coupling parameters we can now show that the heteroclinic cycle loses stability completely in a (degenerate) transverse bifurcation at $r=0$ as both transverse eigenvalues pass through zero.

\begin{thm}\label{thm:StabPiH}
For $\alpha=\ph$ the heteroclinic cycle~$\Cyc_2$ is asymptotically stable if~$r>0$ and completely unstable if $r<0$.
\end{thm}

\begin{proof}
Substituting the stability properties~\eqref{eq:StabDSS}, \eqref{eq:StabDDS}, we obtain
\begin{align}
a &= \frac{c}{e} = \frac{2\K+4r}{2\K-4r} & b &= -\frac{t}{e} = \frac{4r}{2\K-4r}.
\end{align}
Simplifying the expressions $b \geq 0$ and $a + b > 1$ now proves the assertion.
\end{proof}

\subsubsection{Stability for $\alpha\neq\frac{\pi}{2}$}
If~$\alpha\neq\frac{\pi}{2}$, then the two transverse eigenvalues
$t_1 = 2\cos(\al)+4r\cos(2\al)$,
$t_2 = -2\cos(\al)+4r\cos(2\al)$
are distinct. Consequently, there are two transverse bifurcations as the oscillator coupling parameters are varied: we have
\begin{itemize}
\item $t_1, t_2 < 0$ if $r>\frac{\abs{\cos(\al)}}{4\cos^2(\al)-2}$, 
\item $t_1, t_2 > 0$ if $r<\frac{-\abs{\cos(\al)}}{4\cos^2(\al)-2}$, and 
\item one positive and one negative transverse eigenvalue otherwise.
\end{itemize}

The stability of the heteroclinic cycle is now determined by the properties of the transition matrix
\begin{align}\label{trans-matrix}
\M &= \M_\DDS \M_\DSS=\left(\begin{array}{cc}b_1b_2+a_2 & b_1 \\a_1b_2 & a_1\end{array}\right).
\end{align}
The stability calculations are analogous to those for simple heteroclinic cycles in~$\R^4$ of type~$C_2^-$~\cite{Podvigina2011, Lohse2015}; for such cycles, we have the following result.

\begin{lem}[Stability conditions for $C_2^-$ cycles given in~\cite{Podvigina2011, Lohse2015}]
\label{lem:StabC2m}
For asymptotic stability, we need $t_1, t_2 < 0$ and \[\max\sset{b_1b_2 + a_2 + a_1, 2(b_1b_2 + a_2 + a_1-a_1a_2)} > 2.\] 
If $t_2<0<t_1$ and $b_1b_2-a_1+a_2<0$, then the cycle is completely unstable.
\end{lem}

These results can now be used to show that the heteroclinic cycle loses stability completely as one of the transverse eigenvalues becomes positive.

\begin{thm}\label{thm:as-cu}
The heteroclinic cycle~$\Cyc_2$ is asymptotically stable if
\[r>\frac{\abs{\cos(\al)}}{4\cos^2(\al)-2}\]
and completely unstable otherwise.
\end{thm}

\begin{proof}
First, observe that there are relations between the eigenvalues~\eqref{eq:StabDSS}, \eqref{eq:StabDDS} of the linearization at the saddle points. Set $S=2\K\sin(\alpha)$. We have $e_1 = S+t_2$, $c_1 = S-t_2$, $e_2 = S+t_1$, $c_2 = S-t_1$ which are all positive. Consequently, $S>0$ and $c_1 = e_1 - 2t_2$, $c_2 = e_2 - 2t_1$.

If $t_1, t_2 < 0$ (the hypothesis of the theorem is satisfied) we have
\begin{align}
b_1b_2 + a_2 + a_1 &= 
2+\frac{t_1t_2 -2t_1e_1 -2t_2e_2}{e_1e_2}>2
\end{align}
since all terms are positive. Hence by Lemma~\ref{lem:StabC2m}, the heteroclinic cycle is asymptotically stable.

Now suppose that $t_2<0<t_1$ (the case $t_1<0<t_2$ is analogous). We have
\begin{align}
b_1b_2-a_1+a_2 &= 
\frac{t_1t_2 -2St_1+2St_2}{e_1e_2}<0
\end{align}
since all terms are negative. By Lemma~\ref{lem:StabC2m} the heteroclinic cycle is completely unstable.
\end{proof}

The dichotomy between asymptotic stability and complete instability appears to be nongeneric for~$C_2^-$-cycles compared to~\cite[Corollary 4.8]{Lohse2015}. This is due to the fact that~$e_2$ and~$c_2$ are not independent of~$t_1$. In fact, the case $t_1=0$ coincides with the degenerate situation $c_2=e_2$. Therefore, the assumption in~\cite[Corollary 4.8]{Lohse2015} that $b_1b_2-a_1+a_2>0$ even for small $t_1>0$ cannot be satisfied here.

\subsection{Eigenvalues and eigenvectors of the transition matrix products}
\label{sec:ev}

In the previous section we used results from~\cite{Podvigina2011, Lohse2015} (stated as Lemmas~\ref{lem:StabC1m} and~\ref{lem:StabC2m}) to determine the stability of the cycle. We now relate these to the hypotheses in Propositions~\ref{thm:3.4} and~\ref{thm:3.10} by calculating eigenvalues and eigenvectors of the transition matrix products. This is useful for our stability analysis in the higher-dimensional system in Section~\ref{sec:networks}.

For $\alpha \neq \frac{\pi}{2}$ the transition matrix product~$\M$ as defined in~\eqref{trans-matrix} has eigenvalues $\lambda_1>\lambda_2$ given by
\begin{align*}
  \lambda_1&=\frac{1}{2}\left(a_1+a_2+b_1b_2+\sqrt{(a_1-a_2-b_1b_2)^2+4a_1b_1b_2}\right),\\
  \lambda_2&=\frac{1}{2}\left(a_1+a_2+b_1b_2-\sqrt{(a_1-a_2-b_1b_2)^2+4a_1b_1b_2} \right)
\end{align*}
and corresponding eigenvectors
\begin{align*}
  u_1&=(u_{11},u_{12})=\left(1,\frac{a_1-a_2-b_1b_2+\sqrt{(a_1-a_2-b_1b_2)^2+4a_1b_1b_2}}{2b_1}\right),\\
   u_2&=(u_{21},u_{22})=\left(1,\frac{a_1-a_2-b_1b_2-\sqrt{(a_1-a_2-b_1b_2)^2+4a_1b_1b_2}}{2b_1}\right).
\end{align*}
If $t_1,t_2<0$, then both eigenvalues are real and hence condition~\cA~is satisfied. Moreover, by the calculations in the proof of Theorem~\ref{thm:as-cu}, we have
\begin{align*}
  \lambda_1&>\frac{1}{2}(a_1+a_2+b_1b_2)>1,
\end{align*}
so~\cB~is satisfied as well. Since in this case all transition matrices have only nonnegative entries, Proposition~\ref{thm:3.4} applies and the cycle is asymptotically stable. We note that~\cC~is also satisfied, because $4a_1b_1b_2>0$ implies that
\begin{align*}
 u_{11}u_{12}=u_{12}>\frac{a_1-a_2-b_1b_2+|a_1-a_2-b_1b_2|}{2b_1}>0.
\end{align*}
Similarly, for the components of the other eigenvector we get
\begin{align*}
 u_{21}u_{22}=u_{22}<\frac{a_1-a_2-b_1b_2-|a_1-a_2-b_1b_2|}{2b_1}<0.
\end{align*}
This is not directly related to condition~\cC, but will also be used in the following subsection.

On the other hand, if $t_2<0<t_1$, the transition matrix~$\M_1$ has a negative entry. Again by the calculations in the proof of Theorem~\ref{thm:as-cu} we have $a_1-a_2-b_1b_2>0$, and therefore
\begin{align*}
u_{11}u_{12}=u_{12}<\frac{a_1-a_2-b_1b_2}{2b_1}<0.
\end{align*}
Thus, \cC~is violated and by Proposition~\ref{thm:3.10}(a) the cycle is completely unstable. The case $t_1<0<t_2$ is analogous for the other transition matrix product.

\section{Four Coupled Oscillator Populations}\label{sec:networks}

\subsection{Four interacting populations support a heteroclinic network}

In this section we consider $\maxpop=4$ coupled populations with $\maxdim=2$ phase oscillators each. For the coupling function~$g$ as in~\eqref{eq:CN2} and parameter $\dev\in[-1, 1]$ define the interaction functions
\begin{align}\label{eq:CM4}
\tGf(\theta_\tau; \vth) &= -\frac{1}{4}\big(\cos(\theta_{\tau, 1}-\theta_{\tau,2}+\vth+\alpha)+\cos(\theta_{\tau, 2}-\theta_{\tau,1}+\vth+\alpha)\big)\\
\label{eq:NPCM4}\tGt_\sigma(\vth)&=
g(\vth)+\K\left(1-\frac{1}{\maxdim}\right)\K_\sigma\cos(\vth+\alpha)
\end{align}
where $K_1 = 1$, $K_2=-1$, $K_3=-1+\dev$, $K_4=-1-\dev$. Consider the oscillator dynamics where the phase of oscillator~$k$ in population~$\sigma$ evolves according to
{\allowdisplaybreaks%
\begin{subequations}\label{eq:Dyn4x2lin}
\begin{align}
\begin{split}
\dot\theta_{1,k} &= \omega
+\tGt_1(\theta_{1,3-k}-\theta_{1,k})
+\K\tGf(\theta_{4},\theta_{1,3-k}-\theta_{1,k})
\\&\quad 
-\K\tGf(\theta_{2},\theta_{1,3-k}-\theta_{1,k})
+\K\tGf(\theta_{3},\theta_{1,3-k}-\theta_{1,k}),
\end{split}\\
\begin{split}
\dot\theta_{2,k} &= \omega
+\tGt_2(\theta_{2,3-k}-\theta_{2,k})
-\K\tGf(\theta_{4},\theta_{2,3-k}-\theta_{2,k})
\\&\quad 
+\K\tGf(\theta_{1},\theta_{2,3-k}-\theta_{2,k})
-\K\tGf(\theta_{3},\theta_{2,3-k}-\theta_{2,k}),
\end{split}\\
\begin{split}
\dot\theta_{3,k} &= \omega
+\tGt_3(\theta_{3,3-k}-\theta_{3,k})
-\K\tGf(\theta_{4},\theta_{3,3-k}-\theta_{3,k})
\\&\quad 
-\K\tGf(\theta_{1},\theta_{3,3-k}-\theta_{3,k})
+\K(1+\dev)\tGf(\theta_{2},\theta_{3,3-k}-\theta_{3,k}),
\end{split}\\
\begin{split}
\dot\theta_{4,k} &= \omega
+\tGt_4(\theta_{4,3-k}-\theta_{4,k})
-\K\tGf(\theta_{1},\theta_{4,3-k}-\theta_{4,k})
\\&\quad 
+\K(1-\dev)\tGf(\theta_{2},\theta_{4,3-k}-\theta_{4,k}) 
-\K\tGf(\theta_{3},\theta_{4,3-k}-\theta_{4,k}).
\end{split}%
\end{align}%
\end{subequations}%
}%
As shown in Appendix~\ref{app:Reduction}, for $\dev=0$ this is a nonpairwise approximation of the four interacting populations in~\cite{Bick2017c}. The vector field is $(\Sn\times\Tor)^\maxpop$-equivariant as~\eqref{eq:DynMxN}: $\Sn$~acts by permuting oscillators within populations and~$\Torn$ by a phase shift in each population. If $\dev=0$, the system is $(\Sn\times\Tor)^\maxpop\rtimes\Z_2$ equivariant where $\Z_2=\langle(34)\rangle$ acts by permuting populations three and four. If $\delta\neq 0$, then there is a parameter symmetry $(\delta, \theta_3, \theta_4)\mapsto(-\delta, \theta_4, \theta_3)$.

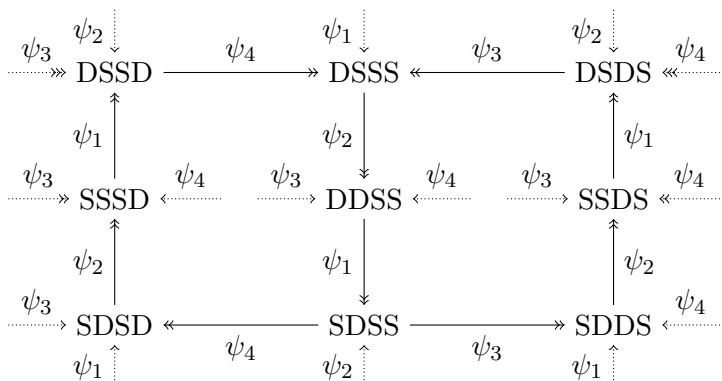
\begin{figure}[t]
{
\begin{tikzpicture}
\tikzstyle{breakd}= [dash pattern=on 23\pgflinewidth off 2pt]
  \matrix (m) [matrix of math nodes, row sep=1.5em,
    column sep=2em, ampersand replacement=\&]{
    \& \PH \& \& \PH \& \& \PH \& \\  
    \PH \& \DSSD \& \& \DSSS \& \& \DSDS \& \PH\\
    \& \& \& \& \& \& \\  
    \PH \& \SSSD \& \PH \& \DDSS \& \PH \& \SSDS \& \PH \\
    \& \& \& \& \& \& \\  
    \PH \& \SDSD \& \& \SDSS \& \& \SDDS \& \PH\\
    \& \PH \& \& \PH \& \& \PH \& \\ };
  \path[-stealth]
    (m-1-4) edge [->,densely dotted] node [left] {$\psi_1$} (m-2-4)
    (m-2-4) edge [->>] node [left] {$\psi_2$} (m-4-4)
    (m-4-3) edge [->,densely dotted] node [above] {$\psi_3$} (m-4-4)
    (m-4-5) edge [->,densely dotted] node [above] {$\psi_4$} (m-4-4)
    (m-4-4) edge [->>] node [left] {$\psi_1$} (m-6-4)
    (m-7-4) edge [->,densely dotted] node [left] {$\psi_2$} (m-6-4)
    
    (m-6-4) edge [->>] node [below] {$\psi_4$} (m-6-2)
    (m-7-2) edge [->,densely dotted] node [left] {$\psi_1$} (m-6-2)
    (m-6-1) edge [->,densely dotted] node [above] {$\psi_3$} (m-6-2)
    (m-6-2) edge [->>] node [left] {$\psi_2$} (m-4-2)
    (m-4-1) edge [->>,densely dotted] node [above] {$\psi_3$} (m-4-2)
    (m-4-3) edge [->,densely dotted] node [above] {$\psi_4$} (m-4-2)
    (m-4-2) edge [->>] node [left] {$\psi_1$} (m-2-2)
    (m-2-1) edge [->>>,densely dotted] node [above] {$\psi_3$} (m-2-2) 
    (m-1-2) edge [->,densely dotted] node [left] {$\psi_2$} (m-2-2)
    (m-2-2) edge [->>] node [above] {$\psi_4$} (m-2-4)

    (m-6-4) edge [->>] node [below] {$\psi_3$} (m-6-6)
    (m-7-6) edge [->,densely dotted] node [left] {$\psi_1$} (m-6-6)
    (m-6-7) edge [->,densely dotted] node [above] {$\psi_4$} (m-6-6)
    (m-6-6) edge [->>] node [right] {$\psi_2$} (m-4-6)
    (m-4-5) edge [->,densely dotted] node [above] {$\psi_3$} (m-4-6)
    (m-4-7) edge [->>,densely dotted] node [above] {$\psi_4$} (m-4-6)
    (m-4-6) edge [->>] node [right] {$\psi_1$} (m-2-6)
    (m-2-7) edge [->>>,densely dotted] node [above] {$\psi_4$} (m-2-6) 
    (m-1-6) edge [->,densely dotted] node [left] {$\psi_2$} (m-2-6)
    (m-2-6) edge [->>] node [above] {$\psi_3$} (m-2-4);
\end{tikzpicture}
}
\caption{\label{fig:KirkSilber}
A heteroclinic network~$\Net_2$ arises in $\maxpop=4$ coupled populations of $\maxdim=2$ oscillators. The types of arrowhead ($>$, $\gg$, $\ggg$)
indicate the eigenvalues for $\alpha=\frac{\pi}{2}$ and $\dev=0$: $\lambda^{\ggg}=-4\K-4r < \lambda^{\gg}=-2\K-4r < \lambda^{>}=-4r < 0 < 2\K-4r$. The $\psi_k$ along an arrow indicate the phase difference that corresponds to the invariant subspace.
}
\end{figure}

\begin{thm}
\label{thm:NetExist}
The system of coupled phase oscillator populations~\eqref{eq:Dyn4x2lin} supports a robust heteroclinic network~$\Net_2$---shown in Figures~\ref{fig:KirkSilber} and~\ref{fig:StabIndNot}---between relative equilibria with localized frequency synchrony.
\end{thm}

\begin{proof}
First, suppose that $\dev=0$. The dynamics on the invariant subspaces~$\pppS$ and $\ppSp$ reduce to~\eqref{eq:Dyn3x2}. Hence by Lemma~\ref{lem:HetCycleM3}, the coupled phase oscillator populations~\eqref{eq:Dyn4x2lin} have a heteroclinic network with two quasi-simple cycles, denoted by $\Cth\subset\pppS$ and $\Ctc\subset\ppSp$. Having~$\dev\neq 0$ constitutes an equivariant perturbation that maintains the $(\Sn\times\Tor)^\maxpop$ symmetry, with respect to which both cycles are robust.
\end{proof}

The eigenvalues of the linearization at the equilibria can be evaluated explicitly. For example,
{\allowdisplaybreaks%
\begin{subequations}\label{eq:StabSDSS}
\begin{align}
\lambda^\SDSS_1 &= -2\K\sin(\al)-2\cos(\al)+4r\cos(2\al),\\
\lambda^\SDSS_2 &= 2\cos(\al)+4r\cos(2\al),\\
\lambda^\SDSS_3 &= 2\K(1+\dev)\sin(\al)-2\cos(\al)+4r\cos(2\al),\\ 
\lambda^\SDSS_4 &= 2\K(1-\dev)\sin(\al)-2\cos(\al)+4r\cos(2\al)
\end{align}
\end{subequations}
}%
determine the stability of the phase configurations~$\SDSS$ in each population respectively. As above, this gives explicit bounds for parameter values which support the heteroclinic network---note that linear stability of the equilibria now also depends on~$\delta$---and conditions to linearize the flow around the heteroclinic network (cf.~Lemma~\ref{lem:Linearizability}).

\begin{figure}[t]
{
\begin{tikzpicture}
\tikzstyle{breakd}= [dash pattern=on 23\pgflinewidth off 2pt]
  \matrix (m) [matrix of math nodes, row sep=2em,
    column sep=1.5em, ampersand replacement=\&]{
    \DSSD \&      \& \DSSS \& 	   \& \DSDS\\
    \SSSD \& \textcolor{gray}{\Ctc} \& \DDSS \& \textcolor{gray}{\Cth} \& \SSDS\\
    \SDSD \&      \& \SDSS \&      \& \SDDS\\};
  \path[-stealth]
    (m-1-3) edge [->] node [left] {$\sic_6$} node [right] {$\sih_6$} (m-2-3)    
    (m-2-3) edge [->] node [left] {$\sic_1$} node [right] {$\sih_1$} (m-3-3)
    
    (m-3-3) edge [->] node [below] {$\sic_2$} (m-3-1)
    (m-3-1) edge [->] node [left]  {$\sic_3$} (m-2-1)
    (m-2-1) edge [->] node [left]  {$\sic_4$} (m-1-1)
    (m-1-1) edge [->] node [above] {$\sic_5$} (m-1-3)

    (m-3-3) edge [->] node [below] {$\sih_2$} (m-3-5)
    (m-3-5) edge [->] node [right] {$\sih_3$} (m-2-5)
    (m-2-5) edge [->] node [right] {$\sih_4$} (m-1-5)
    (m-1-5) edge [->] node [above] {$\sih_5$} (m-1-3);
\end{tikzpicture}
}
\caption{\label{fig:StabIndNot}
The heteroclinic network~$\Net_2$ shown in Figure~\ref{fig:KirkSilber} is constituted by the heteroclinic cycles~$\Cth\subset\pppS$ and $\Ctc\subset\ppSp$. The stability indices along the saddle connections are denoted by~$\sih_q$ and~$\sic_q$, respectively.}
\end{figure}
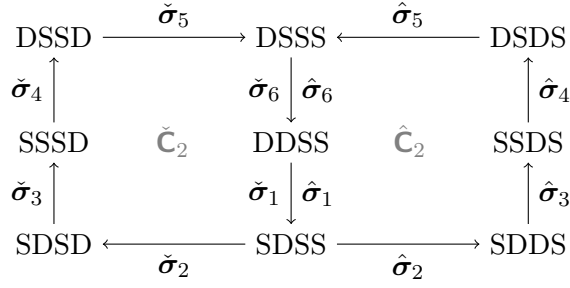

Note that there are other equilibria that are not part of either cycle in the heteroclinic network. For example, on~$\SSSS$ all populations are phase synchronized and its stability is governed by the (quadruple) eigenvalue $\lambda^\SSSS = 4r\cos(2\alpha)-2\cos(\alpha)$. For $\delta=0$ we have $\lambda^\SSSS = \lambda^\DDSS_3$ which implies that~$\SSSS$ is linearly stable if the transverse eigenvalues within the corresponding subspace of each cycle are negative; cf.~Section~\ref{sec:Stability}.

\subsection{Stability of the cycles}

Note that by construction, the saddle~$\SDSS$ has a two-dimensional unstable manifold. Hence, neither cycle can be asymptotically stable for $\dev\approx0$ and $\alpha \approx \frac{\pi}{2}$. Since the cycles are quasi-simple, we can determine their stability by looking at the corresponding transition matrices. Because of the parameter symmetry, we restrict ourselves to the cycle~$\Cth\subset\pppS$ without loss of generality and just write~$\Cyc$ and~$\si_q$ for the remainder of this subsection.

Within the invariant subspace~$\pppS$, we have one contracting, expanding, and transverse direction with local coordinates denoted by $v, w, z$ as above. In addition there is another transverse direction---denoted by~$z^\perp$ in local coordinates---which is mapped to itself under the global map. The second transverse eigenvalues (those transverse to $\pppS$) evaluate to
{\allowdisplaybreaks
\begin{subequations}
\begin{align}
t^\perp_\DSSS &= -2\K\sin(\alpha)-2\cos(\alpha)+4r\cos(2\alpha),\\
t^\perp_\DDSS &= -2\K\delta\sin(\alpha)-2\cos(\alpha)+4r\cos(2\alpha),\\
t^\perp_\SDSS &= 2\K(1-\delta)\sin(\alpha)-2\cos(\alpha)+4r\cos(2\alpha),\\
t^\perp_\SDDS &= -2\K\delta\sin(\alpha)-2\cos(\alpha)+4r\cos(2\alpha),\\
t^\perp_\SSDS &= -2\K\sin(\alpha)-2\cos(\alpha)+4r\cos(2\alpha),\\
t^\perp_\DSDS &= -4\K\sin(\alpha)-2\cos(\alpha)+4r\cos(2\alpha).
\end{align}
\end{subequations}}%
There are two possibilities for transverse bifurcations when~$\delta$ changes. If~$\delta>0$, there is a transverse bifurcation at $t^\perp_\SDSS = 0$. But since $t^\perp_\SDSS = e_\SDSS$ the other cycle of the network then ceases to exist. If $\delta<0$, there is a possibility of two simultaneous transverse bifurcations when $t^\perp_\DDSS = t^\perp_\SDDS = 0$. Write $b^\perp_q = -t^\perp_q/e_q$. Again, the global maps are permutations of the local coordinate axes and the return map evaluates to
\[\h_q(w,z,z^\perp) = (B_q w^{b_{q}}z, D_q w^{a_q}, E_q w^{b_{q}^\perp}z^\perp).\]
where $B_q, D_q, E_q$ are constants.

In logarithmic coordinates $(\eta, \xi, \xi^\perp)$ this gives the transition matrix
\begin{align}\label{3x3-trans-matrix}
\M_q &= \left(\begin{array}{ccc}b_{q} & 1 & 0 \\a_q & 0 & 0 \\b_{q}^\perp & 0 & 1\end{array}\right)
\end{align}
that governs the stability of the cycle. Note that the upper left $2 \times 2$ submatrix is the same as the transition matrix~\eqref{2x2matrix}. In order to simplify notation we write $\xi_1 \  \widehat{=} \ \SDSS$ and $\xi_2,\dotsc, \xi_6$ for the subsequent equilibria of~$\Cyc$. Assuming that we are in a parameter region where the network exists, see Theorem \ref{thm:NetExist}, we can now make the following statement about the stability of its subcycles.
\newcommand{\sX}{\mu}
\newcommand{\sY}{\nu}
\begin{thm}\label{thm:stab-index}
  Assume that the cycle~$\Cyc$ is asymptotically stable\footnote{For $\delta=0$ explicit conditions are given in Theorem~\ref{thm:as-cu}. These can be amended for $\delta\neq 0$ to take the $\delta$-dependency into account.} within the three-dimensional subspace it is contained in and~$\abs{\delta}$ sufficiently small. Then we have the following dichotomy.
\begin{itemize}
\item[(i)] If the transition matrix product~$\M^{(2)}$ satisfies the eigenvector condition~\cC, then~$\Cyc$ is f.a.s.\ and its stability indices~$\si_q$, $q=1,\dotsc,6$, (with respect to the dynamics in $\Tor^\maxpop$) are given by
  \begin{align*}
  \si_q&= F^{\ind}(\sX_q,\sY_q,1)>-\infty,
  \end{align*}
  where $\sX_q=b_q\sX_{q+1}+a_q\sY_{q+1}+b_q^\perp,\ \sY_q=\sX_{q+1}$ for $q=2,\dotsc,6$ and $\sX_1=b_1^\perp,\ \sY_1=0$.
\item[(ii)] If~$\M^{(2)}$ does not satisfy condition~\cC, then~$\Cyc$ is completely unstable.
\end{itemize}
\end{thm}
\begin{proof}
Since $t_1^\perp>0$ is the only positive transverse eigenvalue of an equilibrium in~$\Cyc$, the transition matrix~$\M_1$ is the only one with a negative entry, $b_1^\perp<0$. By Proposition~\ref{thm:3.10} the stability of~$\Cyc$ depends on whether or not all~$\M^{(q)}$ satisfy conditions~\cA--\cC~in Section~\ref{sec:prelim}. Statement~(ii) follows immediately by Proposition~\ref{thm:3.10}(a).

For~(i), we want to apply Proposition~\ref{thm:3.10}(b). By \cite[Lemma 3.6]{Garrido-da-Silva2016} it suffices to show that $\M^{(2)}$ satisfies conditions~\cA~and~\cB, because then all $\M^{(q)}$ satisfy~\cA--\cC. We calculate
\[\M^{(2)}=\M_1\M_6 \dotsb \M_2= \left(\begin{array}{ccc}* & * & 0 \\ * & * & 0 \\b_{1}^\perp \sX+\sY & b_{1}^\perp \tilde{\sX}+\tilde{\sY}  & 1\end{array}\right),\]
where $\sX,\sY,\tilde{\sX},\tilde{\sY}>0$. For a moment, suppose that $\delta=0$. Due to the symmetry of the system in the subspace~$\pppS$, the upper left $2 \times 2$ submatrix is the third power of the matrix~$\M$ in~\eqref{trans-matrix} and we can use our calculations from Section~\ref{sec:ev}. Note that~$\M^{(2)}$ has an eigenvalue $\lambda=1$ with eigenvector $(0, 0, 1)$. Its other two eigenvalues are the third powers of those of~$\M$, call them~$\lambda_1>\lambda_2$, by a slight abuse of notation. Then~$\lmax=\lambda_1>1$ under the assumptions of this theorem, so conditions~\cA~and~\cB~are satisfied. Proposition~\ref{thm:3.10}(b) applies and~$\Cyc$ is f.a.s.. Since eigenvectors and eigenvalues vary continuously in~$\delta$, the same is true for $\abs{\delta}$ sufficiently small.

In order to derive expressions for the stability indices we have to find the arguments~$\beta^{(l)} \in \R^3$ of the function~$F^\ind$ from Proposition~\ref{thm:3.10}. As is shown in~\cite{Garrido-da-Silva2016}, this becomes simpler if for all $q=1,\dotsc ,6$ we have
\begin{align}\label{cond-U}
  U^{-\infty}(\M^{(q)}):=\set{x \in \R^3_-}{\lim\limits_{k \to \infty} \left(\M^{{(q)}}\right)^kx= -\infty}=\R^3_-,
  \end{align}
  where $\R^3_-=\set{(x_1,x_2,x_3)}{x_1,x_2,x_3<0}$ and the convergence is demanded in every component. Clearly, this asymptotic behavior is controlled by the eigenvectors of~$\M^{(q)}$. Consider first the case~$q=2$. Under our assumptions, all components of the eigenvector corresponding to the largest eigenvalue~$\lmax>1$ have the same sign. Another eigenvector is~$(0,0,1)$. From Section~\ref{sec:ev} we know that the first two components of the remaining eigenvector have opposite signs. It follows that any $x\in\R^3_-$ written in the eigenbasis of~$\M^{(2)}$ must have a nonzero coefficient for the largest eigenvector. Therefore, $x\in U^{-\infty}(\M^{(2)})$, so~(\ref{cond-U}) holds. For~$q \neq 2$ note that all~$\M^{(q)}$ are similar, hence they have the same eigenvalues. Their eigenvectors are obtained by multiplying those of~$\M^{(2)}$ by~$\M_2, \M_3\M_2,\dotsc,\M_6\M_5\M_4\M_3\M_2$, respectively. This involves only matrices with nonnegative entries and thus does not affect our conclusions using the signs of the entries of the eigenvectors. Therefore,~(\ref{cond-U}) holds for all~$q=1,\dotsc,6$.

Since~(\ref{cond-U}) is satisfied, the only arguments~$\beta^{(l)} \in \R^3$ that must be considered for~$F^\ind$ in the calculation of~$\si_q$ are the rows of the (products of) transition matrices~$\M_q,\M_{q+1}\M_q,\M_{q+2}\M_{q+1}\M_q$ and so on. Among these, we only need to take rows into account where at least one entry is negative; if there are none, the respective index is equal to~$+\infty$. Negative entries can only occur when~$\M_1$ is involved in the product, and then only in the last row. So for $\si_q$ the last row of $\M_1\M_6\dotsb\M_q$ must be considered. Since~$\M_2$ has no negative entries and its third column is $(0,0,1)$, the first two entries in the last row of $\M_2\M_1\M_6\dotsb\M_q$ are greater than the respective entries of $\M_1\M_6\dotsb\M_q$, yielding a greater value for $F^{\ind}$. The same goes for $\M_3\M_2\M_1\M_6\dotsb\M_q$ and so on. Thus,~$\si_q$ is indeed obtained by plugging the last row of $\M_1\M_6\dotsb\M_q$ into~$F^\ind$. We get
\begin{align*}
  \si_1&=F^{\ind}(\mbox{last row of}\ \M_1),\\
  \si_2&=F^{\ind}(\mbox{last row of}\ \M_1\M_6\M_5\M_4\M_3\M_2),\\
  \si_3&=F^{\ind}(\mbox{last row of}\ \M_1\M_6\M_5\M_4\M_3),\\
  \si_4&=F^{\ind}(\mbox{last row of}\ \M_1\M_6\M_5\M_4),\\
  \si_5&=F^{\ind}(\mbox{last row of}\ \M_1\M_6\M_5),\\
  \si_6&=F^{\ind}(\mbox{last row of}\ \M_1\M_6). 
\end{align*}
The lower right entry of all these matrices is~1, so for all $q=1,\dotsc,6$ we can write $\si_q=F^\ind(\sX_q,\sY_q,1)>-\infty$. Since the last row of~$\M_1$ is~$(b_1^\perp,0,1)$, we have~$\sX_1=b_1^\perp$ and~$\sY_1=0$ as claimed. The recursive relations now follow immediately from~(\ref{3x3-trans-matrix}).
\end{proof}

We conclude this section with a few remarks on these stability results. First, consider condition~\cC. Let $\wmax=(\wmax_1, \wmax_2 ,\wmax_3)$ be the eigenvector of~$\M^{(2)}$ associated with $\lmax$. For $\delta=0$, note that $(\wmax_1, \wmax_2)$ is an eigenvector of~$\M$ associated with its largest eigenvalue, so both of its components have the same sign. To fulfill~\cC, we need $\sgn(\wmax_3)=\sgn(\wmax_{1/2})$. A straightforward calculation yields
\[\wmax_3=\frac{(b_{1}^\perp \sX+\sY)\wmax_1+ (b_{1}^\perp \tilde{\sX}+\tilde{\sY})\wmax_2}{\lmax-1},\]
so it is sufficient to have $\sY>-b_{1}^\perp \sX$ and $\tilde{\sY}>-b_{1}^\perp \tilde{\sX}$. This condition is stronger than assuming~\cC, and as soon as it is satisfied, we have~$\si_2=+\infty$.

By contrast, the indices~$\si_1$ and~$\si_6$ are always finite because~$F^\ind$ has at least one positive and at least one negative argument through~$b_1^\perp$. This makes sense because they are indices along connections shared with the other cycle in the network, while~$\si_2$ belongs to the trajectory that is furthest away from the common ones. For the other indices~$\si_3,\si_4,\si_5$ there is not necessarily a negative argument, so they could be equal to~$+\infty$. From the recursive relations between the~$\sX_q$ and~$\sY_q$ we see that~$\si_q=+\infty$ implies~$\si_{q-1}=+\infty$ for~$q\in\sset{3,4,5}$, which is also plausible in view of the  architecture of the heteroclinic network.

Since $\si_q>-\infty$ for all $q$, we have shown that under the assumptions of Theorem \ref{thm:stab-index} the cycle~$\Cyc$ is not only f.a.s., but indeed attracting a positive measure set along each of its connections. Straightforward constraints on $\sX_q, \sY_q$ given through the definition of $F^\ind$ determine the signs of all $\si_q$ and thus yield necessary and sufficient conditions for $\Cyc$ to be even e.a.s. A simple example for such a necessary condition is $b_1^{\perp}>-1$, so that $\si_1>0$. This is the same as~$e_{\SDSS}>t_{\SDSS}^\perp$ and in terms of the network parameters amounts to~$\dev>0$, cf.~Figure~\ref{fig:StabIndNum}.

\begin{figure}[t]
{
\subfloat[$\delta=0$]{
\begin{tikzpicture}
\tikzstyle{breakd}= [dash pattern=on 23\pgflinewidth off 2pt]
  \matrix (m) [matrix of math nodes, row sep=2em,
    column sep=2em, ampersand replacement=\&]{
    \DSSD \& \DSSS \& \DSDS\\
    \SSSD \& \DDSS \& \SSDS\\
    \SDSD \& \SDSS \& \SDDS\\};
  \path[-stealth]
    (m-1-2) edge [->] node [left] {$0$} node [right] {$0$} (m-2-2)    
    (m-2-2) edge [->] node [left] {$0$} node [right] {$0$} (m-3-2)
    
    (m-3-2) edge [->] node [below] {$+\infty$} (m-3-1)
    (m-3-1) edge [->] node [left]  {$+\infty$} (m-2-1)
    (m-2-1) edge [->] node [left]  {$+\infty$} (m-1-1)
    (m-1-1) edge [->] node [above] {$+\infty$} (m-1-2)

    (m-3-2) edge [->] node [below] {$+\infty$} (m-3-3)
    (m-3-3) edge [->] node [right] {$+\infty$} (m-2-3)
    (m-2-3) edge [->] node [right] {$+\infty$} (m-1-3)
    (m-1-3) edge [->] node [above] {$+\infty$} (m-1-2);
\end{tikzpicture}}
\hfill
\subfloat[$0<\delta\ll 1$]{
\begin{tikzpicture}
\tikzstyle{breakd}= [dash pattern=on 23\pgflinewidth off 2pt]
  \matrix (m) [matrix of math nodes, row sep=2.2em,
    column sep=2em, ampersand replacement=\&]{
    \DSSD \& \DSSS \& \DSDS\\
    \SSSD \& \DDSS \& \SSDS\\
    \SDSD \& \SDSS \& \SDDS\\};
  \path[-stealth]
    (m-1-2) edge [->] node [left] {$-$} node [right] {$+$} (m-2-2)    
    (m-2-2) edge [->] node [left] {$-$} node [right] {$+$} (m-3-2)
    
    (m-3-2) edge [->] node [below] {$+\infty$} (m-3-1)
    (m-3-1) edge [->] node [left]  {$+\infty$} (m-2-1)
    (m-2-1) edge [->] node [left]  {$+$} (m-1-1)
    (m-1-1) edge [->] node [above] {$+$} (m-1-2)

    (m-3-2) edge [->] node [below] {$+\infty$} (m-3-3)
    (m-3-3) edge [->] node [right] {$+\infty$} (m-2-3)
    (m-2-3) edge [->] node [right] {$+\infty$} (m-1-3)
    (m-1-3) edge [->] node [above] {$+\infty$} (m-1-2);
\end{tikzpicture}
}}
\caption{\label{fig:StabIndNum}
Stability indices of the heteroclinic cycles~$\Cth$ and~$\Ctc$ as in Figure~\ref{fig:StabIndNot} for $(\alpha,K,r)=(\frac{\pi}{2}, 0.2, 0.01)$. The symbol `$+$' for a stability index denotes `positive and finite' and `$-$' denotes `negative and finite'.
}
\end{figure}
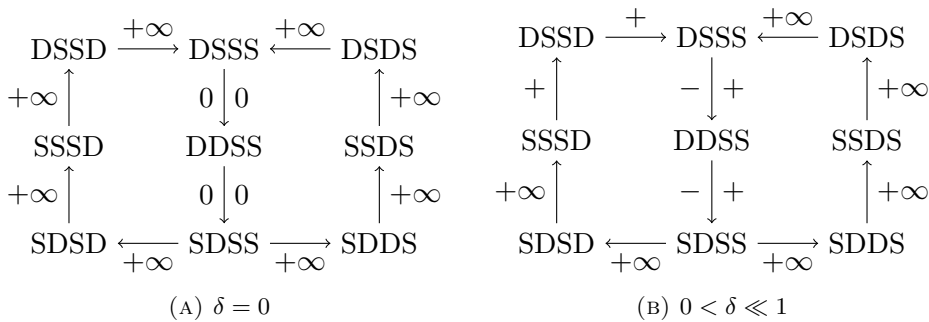

Similar conditions for the other~$\si_q$ become increasingly cumbersome to write down explicitly and we gain little insight from them. Instead, we evaluated the stability indices (of both cycles) numerically. Two cases are illustrated in Figure~\ref{fig:StabIndNum}. We conjecture that there is an open parameter region where the assumptions of Theorem~\ref{thm:stab-index} are satisfied and the network is maximally stable (though not asymptotically stable) due to both cycles being e.a.s.. We comment further on this in the next subsection.

\subsection{Stability of the heteroclinic network}

Even if the stability of all cycles that constitute a heteroclinic network is known, it is hard to make general conclusions about the stability of the network as a whole. For ``simple'' cases, like the Kirk and Silber network~\cite{Kirk1994}, a comprehensive study can be found in~\cite{Castro2014}. Based on the results in the previous section, one can draw several conclusions. If one cycle of~$\Net_2$ is f.a.s.---conditions are given in Theorem~\ref{thm:stab-index}---then the network itself is f.a.s. Moreover, if one cycle, say~$\Cth$, is e.a.s.\ and the heteroclinic trajectories in~$\Ctc$ that are not contained in~$\Cth$ have positive stability indices, then the network is e.a.s.---this is the case in Figure~\ref{fig:StabIndNum}(\sfb).

\newcommand{\sa}{\textrm{s}_\alpha}
\newcommand{\ca}{\textrm{c}_\alpha}

The geometry of the two-dimensional manifold~$\Wu(\SDSS)\subset\SDpp$ gives insight into the dynamics near the heteroclinic network~$\Net_2$. For simplicity, we focus on the case $\alpha=\frac{\pi}{2}$. By~\eqref{eq:Dyn4x2lin}, the dynamics of the phase differences on~$\SDpp$ are given by
\begin{subequations}
\begin{align}
\dot\psi_3 &= \sin(\psi_3)\left(\K\cos(\psi_4)
-4r\cos(\psi_3)
+\K(1+2\delta)
\right)\\
\dot\psi_4 &= \sin(\psi_4)\left(\K\cos(\psi_3)
-4r\cos(\psi_4)
+\K(1-2\delta)
\right).
\end{align}
\end{subequations}
Note that if $\abs{\delta}\K < 2\abs{r}$ there is a (saddle) equilibrium $\xi^{\SDpD}\in\SDpD$ with $\psi_3 = \arccos(\delta\K/2r)\in (0, \pi)$. For the same condition there is an analogous equilibrium $\xi^{\SDDp}\in\SDDp$ with $\psi_4 = \arccos(-\delta\K/2r)\in (0, \pi)$. The stable manifolds of these saddle equilibria now organize the dynamics on~$\SDpp$.

\begin{figure}[t]
{
\subfloat[$\delta=0$]{
\includegraphics[scale=\imagescaling]{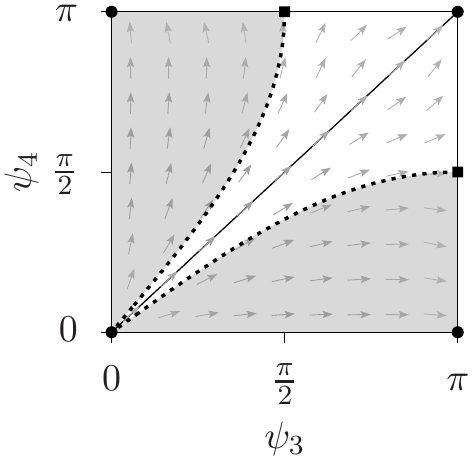}
}\qquad
\subfloat[$\delta=0.07$]{
\includegraphics[scale=\imagescaling]{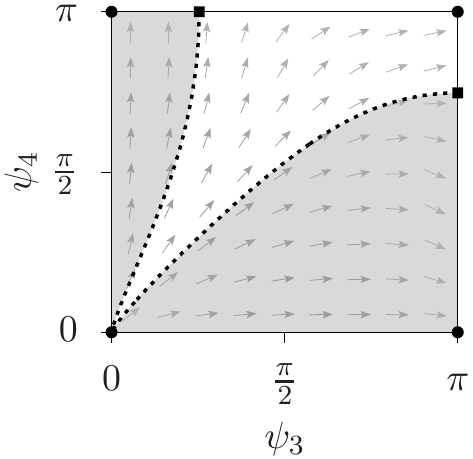}
}}
\caption{\label{fig:WuSDSS}
The two-dimensional unstable manifold $\Wu(\SDSS)\subset\SDpp$ of~$\SDSS$ (bottom left circle) not only contains points (shaded) which are in the stable manifold of~$\DSDS$ (top left circle) and~$\SDSD$ (bottom right circle) but also points in the stable manifold of~$\SDDD$ (top right circle). The stable manifolds of the additional equilibria~$\xi^{\SDpD}$ and~$\xi^{\SDDp}$ (black squares) separate the initial conditions. The other parameters are $(\alpha,K,r)=(\frac{\pi}{2}, 0.2, 0.01)$.
}
\end{figure}

\begin{prop}\label{lem:Wedge}
For the heteroclinic network in Theorem~\ref{thm:NetExist} and parameters $\alpha\approx\frac{\pi}{2}$, $\delta\approx 0$, and an open interval of~$r$ there is robustly an open wedge of initial conditions on~$\Wu(\SDSS)$ such that the trajectories converge to an equilibrium which is not contained in either cycle constituting~$\Net_2$.
\end{prop}

\begin{proof}
We first give conditions on the parameters that ensure that there are no asymptotically stable sets in the invariant set $(0,\pi)^2\subset\SDpp$. It suffices to show that there are no equilibria in $(0,\pi)^2$. By direct calculation, one can verify that for $\alpha=\frac{\pi}{2}$, $\delta=0$ any equilibrium in~$(0,\pi)^2$ must lie in $\sset{\psi:=\psi_3=\psi_4}\subset (0,\pi)^2$. The dynamics of~$\psi$ are given by $\dot\psi=(\K-4r)\cos(\psi)+\K$. Hence, there are no equilibria if $0<r<\K/2$ given $\K>0$; these are exactly the conditions for there to be an asymptotically stable heteroclinic cycle in each subspace by Lemma~\ref{lem:HetCycleM3} and Theorem~\ref{thm:StabPiH}.

Now $\Ws(\xi^{\SDpD})$, $\Ws(\xi^{\SDDp})$ are---as source-saddle connections---robust heteroclinic trajectories $\HC{\SDSS}{\xi^{\SDpD}}$, $\HC{\SDSS}{\xi^{\SDDp}}$. These separate~$(0,\pi)^2$ into three distinct sets of initial conditions which completes the proof.
\end{proof}

The dynamics on $(0, \pi)^2\subset\SDpp$ are shown in Figure~\ref{fig:WuSDSS}. The stable manifolds of $\xi^{\SDpD}$ and~$\xi^{\SDDp}$ subdivide $(0, \pi)^2$ robustly into three wedges with nonempty interior that lie in the stable manifolds of~$\SDDS$, $\SDSD$, and $\SDDD$, respectively. In particular, this suggests that a significant part of trajectories passing by~$\SDSS$ will approach~$\SDDD$ which is not contained in either heteroclinic cycle of the network~$\Net_2$.

\begin{rem}
Let~$\Net$ be a heteroclinic network and let~$\xi_p$, $p=1, \dotsc, P$, denote the equilibria of all its cycles. Abusing the ambiguity of Definition~\ref{defn:HetCycle}%
\footnote{Definition~\ref{defn:HetCycle} is somewhat ambiguous since it does not specify how many heteroclinic connections belong to the heteroclinic network. If~$\Net_2$ only contains one (one-dimensional) heteroclinic trajectory (as suggested by the proof of Theorem~\ref{thm:NetExist}) then it is clearly not almost complete since $\dim(\Wu(\SDSS))=2$. Strictly speaking, for the discussion of completeness we actually consider a network~$\overline{\Net}_2$ that contains the equilibria of~$\Net_2$ and \emph{all} connecting heteroclinic trajectories.}, we call~$\Net$ \emph{complete}~\cite{Ashwin2018} or \emph{clean}~\cite{Field2017} if $\Wu(\xi_p)\subset\bigcup_{q=1}^{P}\Ws(\xi_q)$ for all~$p$ and \emph{almost complete} if the set $\Wu(\xi_p)\cap\bigcup_{q=1}^{P}\Ws(\xi_q)$ is of full measure for all~$p$ and any Riemannian measure on~$\Wu(\xi_p)$; see also~\cite{Bick2018a} for a detailed discussion in the context of coupled oscillator populations. 

For~$\Net_2$ to be almost complete for $\delta\approx 0$, the set $\Wu(\SDSS)\cap(\Ws(\SDDS)\cup\Ws(\SDSD))$ would have to be of full (Lebesgue) measure in~$\SDpp$. However, Lemma~\ref{lem:Wedge} shows that there is a set of nonvanishing measure in~$\SDpp$ which lies in the stable manifold of~$\SDDD$, an equilibrium which is not in the network. Hence, $\Net_2$ cannot be almost complete (nor complete).
\end{rem}

\begin{figure}[t]
\begin{center}
\subfloat[$\delta=0$]{
\includegraphics[scale=\imagescaling]{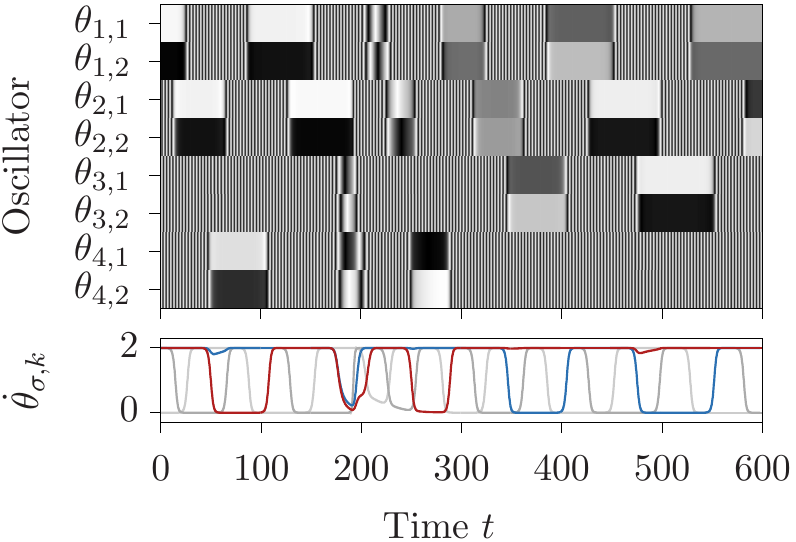}
\ 
\includegraphics[scale=\imagescaling]{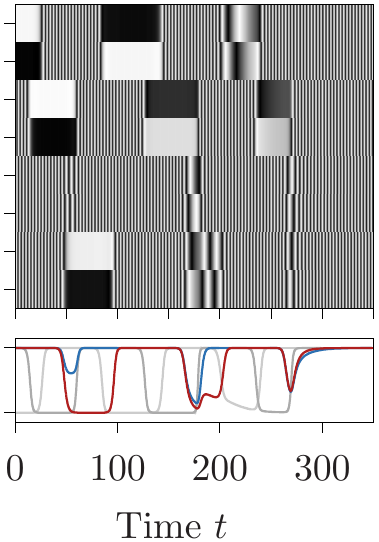}\quad
}\\
\subfloat[$\delta=0.01$]{
\includegraphics[scale=\imagescaling]{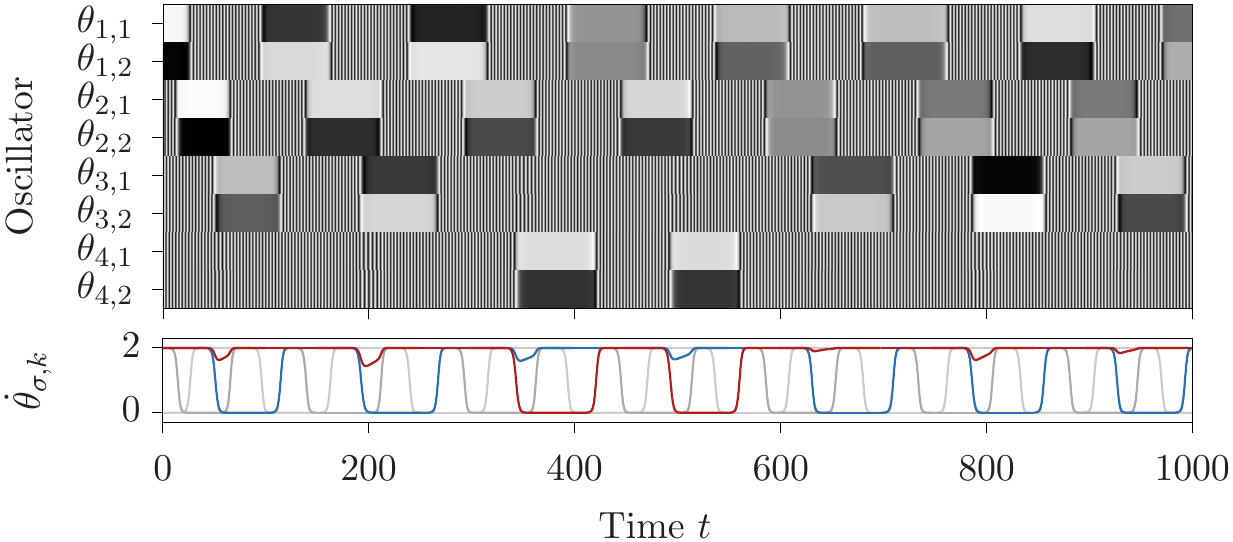}
}
\end{center}
\caption{\label{fig:Dynamics4x2}
The phase dynamics of system~\eqref{eq:Dyn4x2linNum} with noise strength~$\eta=10^{-4}$ shows different realizations of transitions along the heteroclinic network; we have $\alpha=\frac{\pi}{2}$, $r=0.01$, $\K=0.2$ and $\delta=0$ in Panel~(\sfa) and $\delta=0.01$ in Panel~(\sfb). Since the network is not stable, trajectories may go away from the heteroclinic network along the connection $\HC{\SDSS}{\SDDD}$ either to return to a neighborhood of the network (top left) or to converge to the sink~$\SSSS$ (top right).
}
\end{figure}

We further explored the dynamics near the heteroclinic network for $\maxpop=4$ populations of $\maxdim=2$ oscillators using numerical simulations with additive noise. Specifically, for~\eqref{eq:Dyn4x2lin} written as $\dot\theta_{\sigma,k}=\omega+Y_{\sigma,k}(\theta)$---see~\eqref{eq:DynMxN}---and independent Wiener processes~$W_{\sigma,k}$, we solved the stochastic differential equation 
\begin{equation}\label{eq:Dyn4x2linNum}
\dot\theta_{\sigma,k}=\omega+Y_{\sigma,k}(\theta)+\eta W_{\sigma,k}
\end{equation} 
using \textsc{XPP}~\cite{Ermentrout2002}. As shown in Figure~\ref{fig:Dynamics4x2}, the solutions show transitions either along the heteroclinic trajectory $\HC{\SDSS}{\SDDS}$ or $\HC{\SDSS}{\SDSD}$. 
Since the heteroclinic network is not asymptotically stable, the dynamics also show excursions away from the network: there are trajectories that follow the heteroclinic connection $\HC{\SDSS}{\SDDD}$ before either returning to the neighborhood of the network (Figure~\ref{fig:Dynamics4x2}(\sfa) left) or approaching the sink~$\SSSS$ (Figure~\ref{fig:Dynamics4x2}(\sfa) right). If the symmetry is broken by $\delta>0$ (Figure~\ref{fig:Dynamics4x2}(\sfb)), trajectories appear to predominantly follow the principal direction $\HC{\SDSS}{\SDDS}$ with the largest unstable eigenvalue as expected~\cite{Rodrigues2017}.

\section{Discussion}\label{sec:disc}

Coupled populations of identical phase oscillators do not only support heteroclinic networks between sets with localized frequency synchrony but the stability properties of these networks can also be calculated explicitly. Rather than looking at dynamical systems with generic properties at the equilibria, we focus on a specific class of vector fields and obtain explicit expressions for the stability of a heteroclinic network---a feature of the global dynamics of the system---in terms of the coupling parameters. In particular, this does not exclude the possibility that stability properties depend nonmonotonously on the coupling parameters. The coupling parameters themselves can be related to physical parameters of interacting real-world oscillators, for example through phase reductions of neural oscillators~\cite{Hansel1993a}.

Our results motivate a number of further questions and extensions, in particular in the context of the first part of the paper~\cite{Bick2018a}. First, we here restricted ourselves to the quotient system; this is possible by considering nongeneric interactions between oscillator populations. The question what the dynamics look like if the resulting symmetries are broken, will be addressed in future research. Second, what happens for coupled populations with $\maxdim>2$ oscillators? The existence conditions for cycles in~\cite{Bick2018a} and the numerical results in~\cite{Bick2017c} suggest existence of such a network, but stability conditions would rely on the explicit calculation of the stability indices~\cite{Podvigina2011}. In particular, the main tool used here, namely the results for quasi-simple cycles~\cite{Garrido-da-Silva2016}, ceases to apply since the unstable manifold of~$\SDSS$ would be of dimension four and contain points with different isotropy; cf.~\cite{Bick2018a}.

How coupling structure shapes the global dynamics of a system of oscillators is a crucial question in many fields of application. Hence, our results may be of practical interest: in the neurosciences for example, some oscillators may fire at a higher frequency than others for some time before another neural population becomes more active~\cite{Tognoli2014}. The networks here mimic this effect to a certain extent: trajectories which move along the heteroclinic network correspond to sequential speeding up and slowing down of oscillator populations. At the same time, large scale synchrony is thought to relate to neural disfunction~\cite{Uhlhaas2006}. From this point of view, the (in)stability results of Section~\ref{sec:networks} appear interesting, since trajectories in numerical simulations may get ``stuck'' in the fully phase synchronized configuration~$\SSSS$. Hence, our results may eventually elucidate how to design networks that avoid transitions to a highly synchronized pathological state.

\section*{Acknowledgements}

The authors are grateful to P~Ashwin, S~Castro, and L~Garrido-da-Silva for helpful discussions, as well as for mutual hospitality at the Universities of Hamburg and Exeter during several visits where most of this work was carried out.

\appendix

\section{Phase Oscillator Populations with Nonpairwise Coupling}
\label{app:Reduction}

Interaction of phase oscillators through state-dependent phase shift may be approximated by nonpairwise coupling as shown in~\cite{Bick2017c}; here we generalize these calculations to allow for arbitrary coupling topologies.

Consider a system of~$\maxpop$ populations of~$\maxdim$ phase oscillators where $\theta_{\sigma,k}$ denotes the phase of oscillator~$k$ of population~$\sigma$. Recall that the Kuramoto order parameter
\[Z_\sigma = \frac{1}{\maxdim}\sum_{j=1}^\maxdim\exp(i\theta_{\sigma,j})\]
encodes the level of synchrony of population~$\sigma$. In particular, $R_\sigma=\abs{Z_\sigma}=1$ if and only if all oscillators are phase synchronized. Now suppose that the phase of oscillator~$k$ in population~$\sigma$ evolves according to
\begin{equation}\label{eq:DynMxNFull}
\dot\theta_{\sigma,k} = \omega+\sum_{j\neq k} g(\theta_{\sigma,j}-\theta_{\sigma,k}+\CS\Dal_\sigma)
\end{equation}
where the interaction is mediated through the coupling function
\begin{equation}\label{eq:CplngFunc}
g(\vth) = \sin(\vth+\alpha)-r\sin(a(\vth+\alpha)),
\end{equation}
$a\in\N$, and the phase-shifts
\[\Dal_\sigma = \sum_{\tau\neq\sigma} K_{\sigma\tau}(1-R_\tau^2)\]
are linear combination of the (square of the) Kuramoto order parameters. Here~$\K\geq 0$ is the overall interaction strength and $K_{\sigma\tau}\geq 0$ determines the strength of interaction between populations~$\sigma$ and~$\tau$. Set $K_\sigma := \sum_{\tau=1}^\maxpop K_{\sigma\tau}$.

To approximate the dynamics we expand the coupling function. We have
\begin{equation}\label{eq:Subst1}
\begin{split}
g(\vth +\CS\Dal_\sigma) &= g(\vth)+\CS\Dal_\sigma\cos(\vth+\alpha) +O\!\left(\CS r\right) +O\!\left(\CS^2\right).
\end{split}
\end{equation}
Using trigonometric identities we obtain
\[R_\tau^2 = \frac{1}{\maxdim^2}\sum_{p,q=1}^\maxdim \cos(\theta_{\tau,p}-\theta_{\tau,q})\]
which implies
\begin{align}\label{eq:Subst2}
R_\tau^2\cos(\vth+\alpha) &= 
\frac{1}{\maxdim}\cos(\vth+\alpha) +\frac{1}{\maxdim^2}\sum_{p\neq q}\cos(\theta_{\tau,p}-\theta_{\tau,q} + \vth+\alpha).
\end{align}
Now define the interaction functions
\begin{align}
\tGt(\vth) &=
g(\vth)
+\CS\left(1-\frac{1}{\maxdim}\right)\CS_\sigma\cos(\vth+\alpha),\\
\tGf(\theta_\tau; \vth) &=
-\frac{1}{\maxdim^2}\sum_{p\neq q}\cos(\theta_{\tau,p}-\theta_{\tau,q} + \vth+\alpha).
\end{align}
Substituting~\eqref{eq:Subst1} and~\eqref{eq:Subst2} into~\eqref{eq:DynMxNFull} and dropping the $O\!\left(\CS r\right)$, $O\big(\CS^2\big)$ terms yields
\begin{align}\label{eq:DynMxNlin}
\dot\theta_{\sigma,k} &= \omega+
\sum_{j\neq k}\left(
\tGt(\theta_{\sigma,j}-\theta_{\sigma,k})
+K\sum_{\tau=1}^\maxpop K_{\sigma\tau}\tGf(\theta_\tau; \theta_{\sigma,j}-\theta_{\sigma,k})
\right)
\end{align}
as an approximation for~\eqref{eq:DynMxNFull}. Note that the interaction between different populations is through nonpairwise coupling terms: the arguments of the trigonometric functions in~$\tGf$ depend on four phase variables rather than just differences between pairs of phases.

The $\maxpop=3$ coupled populations of $\maxdim=2$ oscillators above are determined by the coupling matrix
\[(K_{\sigma\tau}) = \left(\begin{array}{ccc}
0 & -1 & 1\\
1 & 0 & -1\\
-1 & 1 & 0
\end{array}\right).\]
Note that $\K_\sigma = 0$ for all $\sigma$. Thus,
\begin{align*}
\tGt(\vth) &= g(\vth)\\
\tGf(\theta_\tau; \vth) &= -\frac{1}{4}\big(\cos(\theta_{\tau, 1}-\theta_{\tau,2}+\vth+\alpha)+\cos(\theta_{\tau, 2}-\theta_{\tau,1}+\vth+\alpha)\big)
\end{align*}
and~\eqref{eq:DynMxNlin} reduce to equations~\eqref{eq:Dyn3x2} in the main text above.

Now consider $\maxpop=4$ populations of $\maxdim=2$ oscillators with coupling matrix
\[(K_{\sigma\tau}) = \left(\begin{array}{cccc}
0 & -1 & 1 & 1\\
1 & 0 & -1 & -1\\
-1 & 1+\dev & 0 & -1\\
-1 & 1-\dev & -1 & 0
\end{array}\right)\]
where the parameter~$\dev$ parametrizes the asymmetry between populations three and four. For $\dev=0$ the coupling corresponds to the phase-shifts
\begin{subequations}\label{eq:PhaseShift4x2}
\begin{align}
\Dal_1&=-(1-R_{2}^2)+(1-R_{3}^2)+(1-R_{4}^2),\\
\Dal_2&=(1-R_{1}^2)-(1-R_{3}^2)-(1-R_{4}^2),\\
\Dal_3&=-(1-R_{1}^2)+(1-R_{2}^2)-(1-R_{4}^2),\\
\Dal_4&=-(1-R_{1}^2)+(1-R_{2}^2)-(1-R_{3}^2).
\end{align}
\end{subequations}
considered in~\cite{Bick2017c}. We have
\begin{align}
\tGf(\theta_\tau; \vth) &= -\frac{1}{4}\big(\cos(\theta_{\tau, 1}-\theta_{\tau,2}+\vth+\alpha)+\cos(\theta_{\tau, 2}-\theta_{\tau,1}+\vth+\alpha)\big)\\
\tGt_\sigma(\vth)&=
g(\vth)+\K\left(1-\frac{1}{\maxdim}\right)\K_\sigma\cos(\vth+\alpha)
\end{align}
with $K_1 = 1$, $K_2=-1$, $K_3=-1+\dev$, $K_4=-1-\dev$. Hence, equation~\eqref{eq:Dyn4x2lin} is the nonpairwise approximation~\eqref{eq:DynMxNlin} of the system~\eqref{eq:DynMxNFull}.


\bibliographystyle{unsrt}
\def\urlprefix{}
\def\url#1{}

\bibliography{ref} 

\begin{thebibliography}{10}

\bibitem{Bick2017c}
Christian Bick.
\newblock {Heteroclinic switching between chimeras}.
\newblock {\em Physical Review E}, 97(5):050201(R), 2018.

\bibitem{Bick2018a}
Christian Bick.
\newblock {Heteroclinic Dynamics of Localized Frequency Synchrony: Heteroclinic
  Cycles for Small Populations}.
\newblock {\em In Prep.}, pages 1--23, 2018.

\bibitem{Podvigina2011}
Olga Podvigina and Peter Ashwin.
\newblock {On local attraction properties and a stability index for
  heteroclinic connections}.
\newblock {\em Nonlinearity}, 24(3):887--929, 2011.

\bibitem{Garrido-da-Silva2016}
Liliana {Garrido-da-Silva} and Sofia B. S.~D. Castro.
\newblock {Stability of quasi-simple heteroclinic cycles}.
\newblock {\em Dynamical Systems}, pages 1--31, 2018.

\bibitem{Krupa1997}
Martin Krupa.
\newblock {Robust heteroclinic cycles}.
\newblock {\em Journal of Nonlinear Science}, 7(2):129--176, 1997.

\bibitem{Podvigina2012}
Olga Podvigina.
\newblock {Stability and bifurcations of heteroclinic cycles of type Z}.
\newblock {\em Nonlinearity}, 25(6):1887--1917, 2012.

\bibitem{Melbourne1991}
Ian Melbourne.
\newblock {An example of a nonasymptotically stable attractor}.
\newblock {\em Nonlinearity}, 4:835--844, 1991.

\bibitem{Brannath1994}
Werner Brannath.
\newblock {Heteroclinic networks on the tetrahedron}.
\newblock {\em Nonlinearity}, 7(5):1367--1384, 1994.

\bibitem{Lohse2015}
Alexander Lohse.
\newblock {Stability of heteroclinic cycles in transverse bifurcations}.
\newblock {\em Physica D}, 310:95--103, 2015.

\bibitem{Krupa1995}
Martin Krupa and Ian Melbourne.
\newblock {Asymptotic stability of heteroclinic cycles in systems with
  symmetry}.
\newblock {\em Ergodic Theory and Dynamical Systems}, 15(01):121--147, 1995.

\bibitem{Ashwin1992}
Peter Ashwin and James~W. Swift.
\newblock {The dynamics of n weakly coupled identical oscillators}.
\newblock {\em Journal of Nonlinear Science}, 2(1):69--108, 1992.

\bibitem{Bick2015c}
Christian Bick and Peter Ashwin.
\newblock {Chaotic weak chimeras and their persistence in coupled populations
  of phase oscillators}.
\newblock {\em Nonlinearity}, 29(5):1468--1486, 2016.

\bibitem{Bick2015d}
Christian Bick.
\newblock {Isotropy of Angular Frequencies and Weak Chimeras with Broken
  Symmetry}.
\newblock {\em Journal of Nonlinear Science}, 27(2):605--626, 2017.

\bibitem{Ashwin2014a}
Peter Ashwin and Oleksandr Burylko.
\newblock {Weak chimeras in minimal networks of coupled phase oscillators}.
\newblock {\em Chaos}, 25:013106, 2015.

\bibitem{Aguiar2010a}
Manuela A.~D. Aguiar and Sofia B. S.~D. Castro.
\newblock {Chaotic switching in a two-person game}.
\newblock {\em Physica D}, 239(16):1598--1609, 2010.

\bibitem{Ruelle1989}
David Ruelle.
\newblock {\em {Elements of Differentiable Dynamics and Bifurcation Theory}}.
\newblock Academic Press, New York, NY, 1989.

\bibitem{Field1991}
Michael~J. Field and James~W. Swift.
\newblock {Stationary bifurcation to limit cycles and heteroclinic cycles}.
\newblock {\em Nonlinearity}, 4(4):1001--1043, 1991.

\bibitem{Kirk1994}
Vivien Kirk and Mary Silber.
\newblock {A competition between heteroclinic cycles}.
\newblock {\em Nonlinearity}, 7(6):1605--1621, 1994.

\bibitem{Castro2014}
Sofia B. S.~D. Castro and Alexander Lohse.
\newblock {Stability in simple heteroclinic networks in $\mathbb{R}^4$}.
\newblock {\em Dynamical Systems}, 29(4):451--481, 2014.

\bibitem{Ashwin2018}
Peter Ashwin, Sofia B. S.~D. Castro, and Alexander Lohse.
\newblock {On realizing graphs as complete heteroclinic networks}.
\newblock {\em In Prep.}, 2018.

\bibitem{Field2017}
Michael~J. Field.
\newblock {Patterns of desynchronization and resynchronization in heteroclinic
  networks}.
\newblock {\em Nonlinearity}, 30(2):516--557, 2017.

\bibitem{Ermentrout2002}
Bard Ermentrout.
\newblock {\em {Simulating, Analyzing, and Animating Dynamical Systems}}.
\newblock Society for Industrial and Applied Mathematics, 2002.

\bibitem{Rodrigues2017}
Alexandre A.~P. Rodrigues.
\newblock {Attractors in complex networks}.
\newblock {\em Chaos}, 27(10):103105, 2017.

\bibitem{Hansel1993a}
David Hansel, G.~Mato, and Claude Meunier.
\newblock {Phase Dynamics for Weakly Coupled Hodgkin-Huxley Neurons}.
\newblock {\em Europhysics Letters (EPL)}, 23(5):367--372, 1993.

\bibitem{Tognoli2014}
Emmanuelle Tognoli and J.~A.~Scott Kelso.
\newblock {The Metastable Brain}.
\newblock {\em Neuron}, 81(1):35--48, 2014.

\bibitem{Uhlhaas2006}
Peter~J. Uhlhaas and Wolf Singer.
\newblock {Neural Synchrony in Brain Disorders: Relevance for Cognitive
  Dysfunctions and Pathophysiology}.
\newblock {\em Neuron}, 52(1):155--168, 2006.

\end{thebibliography}

\end{document}